\documentclass[a4paper,final]{siamart171218}

\usepackage{mathrsfs} % \usepackage{rotfig}
\usepackage{mathdots}
\usepackage{amsmath,amsfonts,amssymb}
\usepackage{cleveref}
\usepackage[utf8]{inputenc}
\usepackage{enumitem}
\usepackage{algorithm}
\usepackage{algpseudocode}
\usepackage{pgfplotstable}
\usepackage{tikz}
\usepackage{pgfplots}
\usepackage{scalerel}
\usepackage{multirow}
\newcommand\TheTitle{Rational Krylov for Stieltjes matrix functions: convergence and pole selection}
\newcommand\TheShortTitle{Rational Krylov for Stieltjes functions}
\newcommand\TheAuthors{Stefano Massei, Leonardo	Robol}
\headers{\TheShortTitle}{\TheAuthors}
\title{\TheTitle}

\author{%
	Stefano Massei\thanks{TU Eindhoven, Netherlands,
		\email{s.massei@tue.nl}. The work of Stefano Massei has been partially supported by the SNSF research project \emph{Fast algorithms from low-rank updates}, grant number: 200020\_178806, and by the
		INdAM/GNCS project ``Analisi di matrici sparse e
		data-sparse: metodi
		numerici ed applicazioni''.}  \and Leonardo
	Robol\thanks{Dipartimento di Matematica, Università di Pisa, Italy,
		\email{leonardo.robol@unipi.it}. The work of Leonardo Robol 
		has been partially supported by a GNCS/INdAM project ``Giovani Ricercatori'' 2018.} }

\usepackage{mathrsfs} % \usepackage{rotfig}
\usepackage{mathdots}
\usepackage[utf8]{inputenc}
\usepackage{enumitem}
\usepackage{algorithm}
\usepackage{algpseudocode}
\usepackage{pgfplotstable}
\usepackage{tikz}
\usepackage{pgfplots}
\usepackage{scalerel}
\usepackage{multirow}
\usepackage{hyperref}
\usepackage{pgfplotstable}

\usepackage{mathtools}
\usepackage{tikz}
\usepackage{pgfplots}
\usepackage{pgfplotstable}
\usepackage{geometry}
\usepackage{color}
\usepackage{booktabs}
\usepackage{framed}
\pgfplotsset{compat=1.9}

\pgfplotstableset{
	every head row/.style={before row=\toprule,after row=\midrule},
	clear infinite
}

\usepackage{amsopn}

 \newcommand{\vect} {\mathrm{vec}}

\newcommand{\Si}{\mathrm{Si}}
\newcommand{\si}{\mathrm{si}}
\newcommand{\atan}{\mathrm{atan}}

\newcommand{\pole}{\psi}
\newcommand{\Pole}{\Psi}

\newcommand{\norm}[1]{\lVert #1 \rVert}
\newcommand{\norml}[1]{\left\lVert #1 \right\rVert}
\pgfplotsset{select coords between index/.style 2 args={
		x filter/.code={
			\ifnum\coordindex<#1\fi
			\ifnum\coordindex>#2\fi
		}
}}
\definecolor{lowrankcolor}{rgb}{.75,.75,.75}

\renewcommand{\leq}{\leqslant}
\renewcommand{\geq}{\geqslant}
\renewcommand{\tilde}{\widetilde}

\numberwithin{theorem}{section}

\newsiamremark{remark}{Remark}
\newsiamremark{example}{Example}

\newcommand{\im}{\mathbf i}

\begin{document}
	\maketitle

	\begin{abstract}
Evaluating the action of a matrix function on a vector, that is $x=f(\mathcal M)v$, is an ubiquitous task in applications. When $\mathcal M$ is large, one usually relies on 
Krylov projection methods. In this paper, we
provide effective choices for the poles of the rational Krylov method for approximating $x$ when $f(z)$ is either Cauchy-Stieltjes or Laplace-Stieltjes
(or, which is equivalent, completely monotonic) and $\mathcal M$ is a positive definite 
matrix. 

Relying on the same tools used to analyze the generic situation, 
we then focus on the case $\mathcal M=I \otimes A - B^T \otimes I$, 
and $v$
obtained vectorizing a low-rank matrix; this finds application, for instance, 
in solving fractional diffusion equation on two-dimensional 
tensor grids. We see how to leverage tensorized Krylov subspaces to exploit the Kronecker structure and
we introduce an error analysis for the numerical approximation of $x$.  Pole selection strategies with explicit convergence bounds are given also in this case. 
\end{abstract}

	  \begin{keywords}
	Rational Krylov, Function of matrices, Kronecker sum, Zolotarev problem, Pole selection, 
	Stieltjes functions
\end{keywords}

\section{Introduction}
We are concerned with the evaluation of $x = f(\mathcal M) v$, where $f(z)$ is a 
Stieltjes function, which can be expressed in integral form
\begin{equation}\label{eq:general-stielt}
f(z) = \int_0^\infty g(t, z) \mu(t)\ dt, \qquad 
g(t, z) \in \left\{ 
e^{-tz}, \frac{1}{t + z}
\right\}. 
\end{equation}
The two choices for $g(t, z)$ define \emph{Laplace-Stieltjes} and \emph{Cauchy-Stieltjes} functions, respectively \cite{benzi2017approximation,widder1942laplace}. The former 
class is a superset of the latter, and coincides with the set of 
completely monotonic functions, whose derivatives satisfy $(-1)^j f^{(j)} \geq 0$ over $\mathbb R_+$ 
for any $j \in \mathbb N$. 

We are interested in two instances of this problem; 
first, we consider the case $\mathcal M := A$, 
where $A\in\mathbb C^{n\times n}$
is Hermitian positive definite with
spectrum contained in $[a,b]$, $v\in\mathbb C^{n \times s}$ is a generic
(block)
vector, and a rational Krylov method  \cite{guttel2013rational} is 
used to approximate $x = f(\mathcal M)v$. 
In this case, we want to estimate the Euclidean norm of the error $\norm{x - x_\ell}_2$, 
where $x_\ell$ is the approximation returned by
$\ell$ steps of the method. Second, 
we consider
\begin{equation} \label{eq:2dkronecker}\mathcal M := I \otimes A - B^T \otimes I\in\mathbb C^{n^2\times n^2},
\end{equation}
where $A,-B\in\mathbb C^{n\times n}$ are Hermitian positive definite with spectra contained in $[a,b]$, $v = \vect(F)\in\mathbb C^{n^2}$ 
is the vectorization of a low-rank matrix $F = U_F V_F^T\in\mathbb C^{n\times n}$, and a tensorized rational Krylov method 
\cite{benzi2017approximation}
is 
used for computing $\vect(X) = f(\mathcal M) \vect(F)$. This
problem is a generalization of the solution of 
a Sylvester equation with a low-rank right hand side, which corresponds to evaluate the function
$f(z) = z^{-1}$. 
Here, we are concerned with estimating the
quantity $\norm{X - X_\ell}_2$, where $X_\ell$ is 
the approximation obtained after $\ell$ steps.  

\subsection{Main contributions}

This paper discusses the connection between 
rational Krylov evaluation of Stieltjes matrix functions
and the parameter
dependent
rational approximation (with the given poles)
of the kernel functions $e^{-tz}$ and 
$\frac{1}{t+z}$. 

The contributions of this work are the following:	
\begin{enumerate}
	\item Corollary~\ref{cor:lapl-stielt} provides a choice of poles 
	for the rational Krylov approximation of $f(\mathcal M)v$, where $f(z)$ 
	is Laplace-Stieltjes, with an explicit error bound depending 
	on the spectrum of $A$. 
	\item Similarly, for Cauchy-Stieltjes functions, we show (in Corollary~\ref{cor:cauchy-stielt}) how leveraging 
	an approach proposed in \cite{druskin2009} allows to 
	recover a result previously given in \cite{beckermann2009error}
	using different theoretical tools. 
	\item In Section~\ref{sec:nested}, we obtain new nested sequences of poles by applying the approach of equidistributed sequences to the results in Corollary~\ref{cor:lapl-stielt}--\ref{cor:cauchy-stielt}. 
	\item In the particular case where $\mathcal M := I \otimes A - B^T \otimes I$ we extend the analysis recently proposed in \cite{benzi2017approximation} to rational Krylov subspaces. Also
	in this setting, we provide explicit choices for the poles 
	and explicit convergence bounds. For Laplace-Stieltjes functions
	a direct consequence of the analysis mentioned above leads to Corollary~\ref{cor:stieltjes-lapl-posdef}; in the Cauchy case, 
	we describe a choice of poles that enables the simultaneous solution
	of a set of parameter dependent Sylvester equations. This results
	in a practical choice of poles and an explicit error bound given 
	in Corollary~\ref{cor:stieltjes-zolotarev-posdef}. 
	\item Finally, we give results predicting the decay in the singular 
	values of $X$ where
	$\vect(X) = f(\mathcal M) \vect(F)$, $F$ is a low-rank matrix, 
	and $f(z)$ is either Laplace-Stieltjes (Theorem~\ref{thm:singvals-lapl-decay}) or 
	Cauchy-Stieltjes (Theorem~\ref{thm:singdecay}). This generalizes 
	the well-known low-rank approximability 
	of the solutions of 
	of Sylvester equations with low-rank right hand
	sides \cite{Beckermann2019}. The result for Laplace-Stieltjes follows
	by the error bound for the rational Krylov method and an Eckart-Young argument. The one for Cauchy-Stieltjes requires to combine the integral representation
	with the ADI approximant for the solution of matrix equations. 
\end{enumerate}
The error bounds obtained are summarized in Table~\ref{tab:summary}. 

We recall that completely monotonic functions
are well approximated by exponential sums \cite{braess2012nonlinear}.
Another consequence of our results in the Laplace-Stieltjes case 
is to constructively show 
that they are also well-approximated 
by rational functions. 

\begin{table}
	\centering 
	\caption{Summary of the convergence rates
		for rational Krylov methods with
		the proposed poles. The convergence rate $\rho_{[\alpha,\beta]}$ is defined by $\rho_{[\alpha,\beta]} := \exp(
		-\pi^2 /\log(4\frac{\beta}{\alpha})
		)$. }
	\label{tab:summary}
	\vskip 10pt
	\begin{tabular}{c|ccc}
		Function class & Argument & Error bound & Reference \\ \hline 
		\multirow{2}{*}{Laplace-Stieltjes} & $\mathcal M := A$ & $\norm{x - x_\ell}_2 \sim \mathcal O(\rho_{[a,b]}^{\frac{\ell}{2}})$ & Cor.~\ref{cor:lapl-stielt} \\
		& $\mathcal M := I \otimes A - B^T \otimes I$ & $\norm{X - X_\ell}_2 \sim \mathcal O(\rho_{[a,b]}^{\frac{\ell}{2}})$ & Cor.~\ref{cor:stieltjes-lapl-posdef} \\ \hline 
		\multirow{2}{*}{Cauchy-Stieltjes} & $\mathcal M := A$ & $\norm{x - x_\ell}_2 \sim \mathcal O(\rho_{[a,4b]}^{\ell})$ & Cor.~\ref{cor:cauchy-stielt} \\
		& $\mathcal M := I \otimes A - B^T \otimes I$ &
		$\norm{X - X_\ell}_2 \sim \mathcal O(\rho_{[a,2b]}^{\ell})$
		& Cor.~\ref{cor:stieltjes-zolotarev-posdef} \\ \hline 
	\end{tabular}
\end{table}

%	It is well known that completely monotonic functions
%	are well approximated by exponential sums \cite{braess2012nonlinear}. Our results prove constructively
%	that they are also well-approximated 
%	by rational functions. 
%	
%	A byproduct of our analysis is that, in the Kronecker
%	structured case, the low-rank property of the right
%	hand side  numerically guarantees the low-rank
%	approximability of the matrix $X$. This generalizes 
%	the well-known property of the solutions of 
%	of Sylvester equations with low-rank right hand
%	sides \cite{Beckermann2019}. In particular, 
%	we provide bounds for the exponential decay for the
%	singular values of $X$, see Section~\ref{sec:low-rank-approximability}.

\subsection{Motivating problems}

Computing the action of a matrix function on
a vector is a classical task in numerical
analysis, and finds applications in 
several fields, such as complex networks \cite{benzi2013total}, 
signal processing \cite{susnjara2015accelerated}, numerical solution of ODEs \cite{hochbruck2010exponential}, and 
many others. 

Matrices with the Kronecker sum structure
as in \eqref{eq:2dkronecker}
often arise from the discretization of differential
operators 
on tensorized 2D grids. 
Applying the inverse of such matrices to a vector
is equivalent to solving a matrix equation. 
When the right hand side is a smooth function or 
has small support, the vector $v$ is the
vectorization of a numerically low-rank matrix. 
The latter property has been exploited to develop several efficient solution methods, see  \cite{simoncini2016computational} and the references
therein. Variants of these approaches have been proposed under weaker
assumptions, such as when the
smoothness is only available far from the diagonal $x = y$, as it happens
with kernel functions \cite{kressner2019low,massei2017solving}. 

In recent years, there has been an increasing interest
in models involving fractional derivatives. 
For 2D problems on rectangular grids, discretizations
by finite differences or finite elements lead to
linear systems that can be recast as the
solution of matrix equations with
particularly structured coefficients \cite{breiten2016low-rank,massei2019fast}.
However, a promising formulation which simplifies the 
design of boundary conditions relies on first
discretizing the 2D Laplacian on the chosen
domain, and then considers the action of the
matrix function $z^{-\alpha}$ (with the Laplacian
as argument) on the right hand side. This is known in the literature as the \emph{matrix transform method} \cite{yang2011novel}. In this framework, one 
has $0 < \alpha < 1$, and therefore $z^{-\alpha}$ is 
a Cauchy-Stieltjes function, a property that has been previously exploited 
for designing fast and stable restarted polynomial Krylov methods for its evaluation \cite{schweitzer2016restarting}. The algorithm proposed in this paper allows to exploit the Kronecker structure of the 2D Laplacian on
rectangular domains in the evaluation of the matrix function.  

Another motivation for our analysis
stems from the study of exponential integrators, where
it is required to evaluate the $\varphi_j(z)$ functions 
\cite{hochbruck2010exponential}, which are part of the Laplace-Stieltjes class. This 
has been the subject of deep studies concerning (restarted) polynomial and rational Krylov methods \cite{frommer2014efficient,schweitzer2016restarting}. 
However, to the best of our knowledge the Kronecker structure, 
and the associated low-rank preservation, has not been exploited in these approaches, 
despite being often present in discretization of differential operators \cite{townsend2015automatic}.

The paper is organized as follows. In Section~\ref{sec:stieltjes-functions} we recall the 
definitions and main properties of Stieltjes functions. Then, 
in Section~\ref{sec:stieltjes} we recall the rational Krylov 
method and then we analyze the simultaneous approximation
of parameter dependent exponentials and resolvents; this leads
to the choice of poles and convergence bounds for Stieltjes 
functions given in Section~\ref{sec:conv1d}. 
In Section~\ref{sec:conv2d} we provide an analysis of the
convergence of the method proposed in \cite{benzi2017approximation}
when rational Krylov subspaces are employed. In particular, in Section~\ref{sec:low-rank-approximability} 
we provide decay bounds for 
the singular values of $X$ such that $\vect(X)  = f(\mathcal M) \vect(F)$. We give some concluding remarks and outlook in 
Section~\ref{sec:conclusions}.     

\section{Laplace-Stieltjes and  Cauchy-Stieltjes functions}
\label{sec:stieltjes-functions}
We recall the definition  and the properties of Laplace-Stieltjes and 
Cauchy-Stieltjes functions
that are relevant for our analysis.  
Functions expressed as Stieltjes integrals
admit a representation of the form:
\begin{equation} \label{eq:stieltjes}
f(z) = \int_{0}^\infty g(t, z) \mu(t) \ dt, 
\end{equation}
where $\mu(t) dt$ is a (non-negative) 
measure on $[0, \infty]$, and $g(t, z)$ is integrable 
with respect to that measure. The choice of $g(t,z)$ determines the particular class of
Stieltjes functions under consideration (\emph{Laplace-Stieltjes} or \emph{Cauchy-Stieltjes}), 
and $\mu(t)$ is called the {density} of $f(z)$. $\mu(t)$ 
can be a proper function, or a distribution, e.g. a Dirac delta. In particular, we can restrict the domain of integration to a subset of $(0,\infty)$ imposing that $\mu(t)=0$ elsewhere. We refer the reader to \cite{widder1942laplace}
for further details.

\subsection{Laplace-Stieltjes functions} 
Laplace-Stieltjes functions are obtained by setting  $g(t,z) = e^{-tz}$ in \eqref{eq:stieltjes}. 

\begin{definition} \label{def:laplace-stieltjes}
	Let $f(z)$ be a function defined on 
	% a subset of $\mathbb C \setminus \mathbb{R}_-$. 
	$(0, +\infty)$. 
	Then, $f(z)$ is a \emph{Laplace-Stieltjes} function 
	if there is a positive measure $\mu(t)dt$  on
	$\mathbb R_+$
	such that 
	\begin{equation}\label{eq:laplace-func}
	f(z) = \int_0^{\infty}  e^{-tz} \mu(t)\ dt. 
	\end{equation}
\end{definition}
Examples of Laplace–Stieltjes functions include: 
\begin{align*}
f(z) &= z^{-1} =  \int_0^\infty
e^{-tz}t\ dt, &
f(z) &= e^{-z}=\int_1^\infty
e^{-tz}\ dt, \\
f(z) &= (1-e^{-z})/z = \int_0^\infty
e^{-tz}\mu(t)\ dt,& \mu(t):=\begin{cases}
1 & t\in[0,1]\\
0&t> 1
\end{cases}.
\end{align*}
The last example is an instance of a particularly relevant
class of Laplace-Stieltjes functions, 
with applications
to exponential integrators. These are often denoted by $\varphi_j(z)$, and 
can be defined as follows:
\[
\varphi_j(z) := \int_0^\infty e^{-tz} \frac{\left[\max\{1-t, 0\}\right]^{j-1}}{(j-1)!}\ dt, \qquad 
j \geq 1. 
\]

A famous theorem of Bernstein states the equality between the set of Laplace–Stieltjes functions and those of \emph{completely monotonic
	functions in $(0, \infty)$} \cite{bernstein}, that is 
$(-1)^j f^{(j)}(z)$ is positive over $(0, \infty)$ for any $j\in\mathbb N$.

From the algorithmic point of view, the explicit knowledge of 
the Laplace density $\mu(t)$ will not play any role. Therefore, 
for the applications of the algorithms and projection methods described
here, it is only relevant to know that a function is in this class. 

\subsection{Cauchy-Stieltjes functions}

Cauchy-Stieltjes functions form a subclass of Laplace-Stieltjes functions, and 
are obtained from
\eqref{eq:stieltjes} by setting $g(t,z) = (t+z)^{-1}$. 

\begin{definition} \label{def:cauchy-stieltjes}
	Let $f(z)$ be a function defined on $\mathbb C \setminus \mathbb{R}_-$. 
	Then, $f(z)$ is a \emph{Cauchy-Stieltjes} function 
	if there is a positive measure $\mu(t)dt$ on
	$\mathbb R_+$ 
	such that 
	\begin{equation}\label{eq:cauchy-func}
	f(z) =  \int_0^{\infty} \frac{\mu(t)}{t+z}\ dt. 
	\end{equation}
\end{definition}

A few examples of Cauchy-Stieltjes functions are: 
\begin{align*}
f(z) &= \frac{\log(1 + z)}{z} =  \int_1^\infty
\frac{t^{-1}}{t + z}\ dt, &
f(z) &= \sum_{j = 1}^h \frac{\alpha_j}{z - \beta_j}, \qquad \alpha_j>0,\quad
\beta_j < 0, \\
f(z) &= z^{-\alpha} = \frac{\sin(\alpha \pi)}{\pi} \int_0^\infty
\frac{t^{-\alpha}}{t + z}\ dt, & \alpha \in (0, 1).
\end{align*}
The rational functions with poles on the negative real semi-axis 
do not belong to this class if one requires $\mu(t)$ to be
a function, but they can be obtained by setting $\mu(t) =
\sum_{j = 1}^h \alpha_j \delta(t - \beta_j)$, where $\delta(\cdot)$ 
is the Dirac delta with unit mass at $0$. For instance, 
$z^{-1}$ is obtained setting $\mu(t) := \delta(t)$. 

Since  Cauchy-Stieltjes functions are also completely monotonic in $(0,\infty)$ \cite{berg}, the set of  Cauchy-Stieltjes functions is contained in the one of Laplace-Stieltjes functions. Indeed, assuming that
$f(z)$ is a Cauchy-Stieltjes function with density $\mu_C(t)$, 
one can construct a Laplace-Stieltjes representation as follows:
\[
f(z) = \int_0^\infty \frac{\mu_C(t)}{t+z}\ dt = 
\int_0^\infty \int_0^\infty e^{-s(t+z)} \mu_C(t)\ ds\ dt = 
\int_0^\infty e^{-sz} \underbrace{\int_0^\infty e^{-st} \mu_C(t)\ dt}_{\mu_L(s)}\ ds, 
\]
where $\mu_L(s)$ defines the Laplace-Stieltjes density. In particular, 
note that if $\mu_C(t)$ is positive, so is $\mu_L(s)$. 
For a more detailed analysis
of the relation between Cauchy- and Laplace-Stieltjes functions we refer
the reader to  \cite[Section~8.4]{widder1942laplace}. 

As in the Laplace case, the explicit knowledge of 
$\mu(t)$ is not crucial for the analysis and is not
used in the algorithm. 
\section{Rational Krylov for evaluating Stieltjes functions}
\label{sec:stieltjes}
Projection schemes for the evaluation of the quantity $f(A) v$ 
work as follows: an orthonormal basis $W$ for a (small) subspace $\mathcal W\subseteq\mathbb C^{n}$ 
is computed, together with the projections  $A_{\mathcal W}:=W^* A W $ and 
$v_{\mathcal W} := W^* v$. Then, the action of $f(A)$ on $v$ is approximated 
by:
\[
f(A) v \approx x_{\mathcal W}:=W f(A_{\mathcal W}) v_{\mathcal W}.
\]
Intuitively, the choice of the subspace $\mathcal W$ is crucial for the quality of the approximation. Usually, one is interested in providing a sequence of subspaces $\mathcal W_1\subset \mathcal W_2\subset \mathcal W_3\subset\dots$ and study the convergence of $x_{\mathcal W_j}$ to $f(A)v$ as $j$ increases.  A common choice for the space $\mathcal W_j$ are Krylov subspaces. 

\subsection{Krylov subspaces}
\label{sec:krylov-spaces}

Several functions can be accurately approximated
by polynomials. 
The  idea behind the standard Krylov method is to generate a subspace that contains  all the quantities of the form $p(A)v$ for every $p(z)$  polynomial of bounded degree. 

\begin{definition}
	Let $A$ be an $n\times n$ matrix, and $v \in \mathbb C^{n \times s}$ be a (block) vector. The \emph{Krylov subspace of order $\ell$} generated by $A$ and $v$ is defined as 
	\[
	\mathcal K_\ell(A, v) := \mathrm{span}\{ 
	v, Av, \ldots, A^{\ell} v
	\}= \{p(A)v:\ \deg(p)\leq\ell\}. 
	\]
\end{definition}

Projection on Krylov subspaces is closely related to polynomial approximation. Indeed, if $f(z)$ is well approximated by $p(z)$, then $p(A)v$ is a good approximation of $f(A)v$, in the sense that 
$\norm{f(A)v - p(A)v}_2 \leq \max_{z \in [a,b]}
|f(z) - p(z)| \cdot \norm{v}_2$.

Rational Krylov subspaces
are their rational analogue, and can be defined as follows. 

\begin{definition}
	Let $A$ be a $n \times n$ matrix, $v \in \mathbb C^{n \times s}$ be a (block) vector and
	$\Pole = (\pole_1, \ldots, \pole_\ell)$, with $\pole_j \in \overline{\mathbb C}:=\mathbb C \cup \{ \infty \}$. The \emph{rational
		Krylov subspace} with poles $\Pole$ generated by $A$ and $v$ is defined as
	\[
	\mathcal{RK}_{\ell}(A, v, \Pole) = 
	q_{\ell}(A)^{-1} \mathcal K_{\ell}(A, v) = \mathrm{span}\{
	q_{\ell}(A)^{-1} v, q_{\ell}(A)^{-1} A v, \ldots, 
	q_{\ell}(A)^{-1} A^{\ell} v
	\}, 
	\]
	where $q_\ell(z) = \prod_{j = 1}^\ell (z - \pole_j)$ and if $\pole_j = \infty$, 
	then we omit the factor $(z - \pole_j)$ from $q_\ell(z)$.
\end{definition}

Note that, a Krylov subspace is a particular rational Krylov subspace
where all poles are chosen equal to $\infty$:
$\mathcal{RK}_\ell(A, v, (\infty, \ldots, \infty)) = \mathcal K_{\ell}(A, v)$. 	 
A common strategy of pole selection consists in alternating $0$ and $\infty$. The resulting vector space is known in the literature as the \emph{extended Krylov subspace} \cite{druskin1998extended}.

We denote by $\mathcal P_\ell$  the set  of polynomials of degree at most $\ell$,  
and by $\mathcal R_{\ell,\ell}$ the set of rational functions $g(z)/l(z)$ with $g(z),l(z)\in\mathcal P_{\ell}$. Given $\Pole=\{\pole_1, \ldots, \pole_\ell\}\subset\overline{\mathbb  C}$, we indicate with $\frac{\mathcal P_\ell}{\Pole}$ the set of rational functions 
of the form $g(z)/l(z)$, with $g(z) \in \mathcal P_\ell$ and $l(z):=\prod_{\pole_j\in\Pole\setminus\{\infty\}}(z-\pole_j)$. 

It is well-known that Krylov subspaces contain the action of related rational matrix functions of $A$ on the (block) vector $v$, if the poles of the
rational functions are a subset of the poles used to construct the approximation
space.
\begin{lemma}[Exactness property]\label{lem:exact}
	Let $A$ be a $n \times n$ matrix, $v \in \mathbb C^{n \times s}$ be a (block) vector and
	$\Pole = \{\pole_1, \ldots, \pole_\ell\}\subset \overline{\mathbb C}$. If $U_{\mathcal P},U_{\mathcal R}$ are orthonormal bases of $\mathcal{K}_{\ell}(A, v)$ and $\mathcal{RK}_{\ell}(A, v, \Pole)$, respectively, then:
	\begin{enumerate}
		\item $f(z)\in\mathcal P_{\ell}\quad\Longrightarrow\quad f(A)v=U_{\mathcal P}f(A_{\ell})(U_{\mathcal P}^*v)\in \mathcal K_{\ell}(A, v),\quad A_\ell=U^*_{\mathcal P}AU_{\mathcal P} $,
		\item $f(z)\in \frac{\mathcal P_{\ell}}{\Pole} \quad\Longrightarrow\quad f(A)v=U_{\mathcal R}f(A_{\ell})(U_{\mathcal R}^*v)\in \mathcal {RK}_{\ell}(A, v,\Pole)$, $\quad A_\ell=U^*_{\mathcal R}AU_{\mathcal R} $. 
	\end{enumerate} 
\end{lemma}

Lemma~\ref{lem:exact} enables to prove the quasi-optimality of the Galerkin projection described 
in Section~\ref{sec:stieltjes}. Indeed, if 
$\mathcal W := \mathcal{RK}(A,v,\Pole)$, then 
\cite{guttel2013rational}
\begin{equation} \label{eq:galerkin-optim}
\norm{x_{\mathcal W} - x}_2 \leq 2 \cdot \norm{v}_2 \cdot \min_{r(z)\in \frac{\mathcal P_\ell}{\Pole}} \max_{z \in [a,b]} |f(z) - r(z)|. 
\end{equation}

The optimal choice of poles for generating the rational  Krylov subspaces is problem dependent and linked to the rational approximation of the function $f(z)$ on $[a,b]$. We investigate how to perform this task when $f$ is either a  Laplace-Stieltjes or Cauchy-Stieltjes function. 

\subsection{Simultaneous approximation of resolvents and matrix exponentials}
The integral expression \eqref{eq:general-stielt}
reads as
\[
f(A) v = \int_0^\infty g(t, A) \mu(t)\ dt, \qquad 
g(t,A) \in \{ 
e^{-tA}, (tI + A)^{-1}
\}
\]
when evaluated at a matrix argument. Since the projection is a linear operation we have
\[
x_{\mathcal W}=W f(A_{\mathcal W}) v_{\mathcal W}=\int_0^\infty Wg(t,A_{\mathcal W}) v_{\mathcal W}\ \mu(t)dt.
\]
This 
suggests to look for
a space approximating uniformly well, in the 
parameter $t$, matrix exponentials
and resolvents, respectively. 
A result concerning the approximation
error in the $L^2$ norm for $t \in \im \mathbb R$ is given in 
\cite[Lemma~4.1]{druskin2009}. 
The proof is obtained exploiting some results on the skeleton approximation
of $\frac{1}{t+\lambda}$ \cite{oseledets2007lower}. We provide a pointwise
error bound, which can be obtained by  following the same steps of the proof 
of \cite[Lemma~4.1]{druskin2009}. We include 
the proof for completeness. 

\begin{theorem}\label{thm:lapl-zol}
	Let $A$ be Hermitian positive definite with spectrum contained in $[a,b]$ and $U$ be an orthonormal basis of $\mathcal{U_R}=\mathcal{RK}_\ell(A,v,\Pole)$. 
	Then, $\forall t\in[0,\infty)$, we have the following inequality:
	\begin{equation}
	\label{eq:sim-ris}
	\norm{(tI+A)^{-1}v-U(tI+A_\ell)^{-1} v_\ell}_2\leq \frac{2}{t+a} \norm{v}_2 \min_{r(z)\in\frac{\mathcal P_{\ell}}{\Pole}} \frac{\max_{z\in[a,b]}|r(z)|}{\min_{z\in(-\infty,0]}|r(z)|} 
	\end{equation}
	where  $A_\ell=U^*AU$ and $ v_\ell=U^*v$. 
\end{theorem}
\begin{proof}
	Following the construction in \cite{oseledets2007lower}, 
	we consider the function $f_{\mathrm{skel}}(\lambda, t)$ defined by 
	\[
	f_{\mathrm{skel}}(t,\lambda) := \begin{bmatrix}
	\frac{1}{t_1 + \lambda} & \ldots &  		   \frac{1}{t_\ell + \lambda}
	\end{bmatrix} M^{-1} \begin{bmatrix}
	\frac{1}{t+\lambda_1} \\ \vdots \\ \frac{1}{t+\lambda_\ell}
	\end{bmatrix}, \qquad M_{ij} = \frac{1}{t_j + \lambda_i}\in\mathbb C^{\ell\times \ell}
	,
	\]
	where $M_{ij}$ are the entries of $M$ and $(t_j, \lambda_i)$ is  an $\ell\times \ell$ grid of interpolation nodes. The function $f_{\mathrm{skel}}(t,\lambda)$ is usually called Skeleton approximation and it is practical for approximating $\frac{1}{t+\lambda}$; indeed its relative error takes the explicit form: $ 1 - (t+\lambda)f_{\mathrm{skel}}(t,\lambda) = \frac{r(\lambda)}{r(-t)}$ with $
	r(z) = \prod_{j = 1}^\ell \frac{z - \lambda_j}{t_j + z}$. If 
	this ratio of rational functions is small, then
	$f_{\mathrm{skel}}(t,\lambda) $
	is a good approximation of $\frac{1}{t+\lambda}$ and --- consequently --- $f_{\mathrm{skel}}(t,A)$ is a good approximation of $(tI+A)^{-1}$. Note that, for every fixed $t$, $f_{\mathrm{skel}}(t,\lambda) $ is a rational function in $\lambda$ with poles $-t_1, \ldots, -t_\ell$. Therefore, using the
	poles $\pole_j=-t_j$, $j=1,\ldots,\ell$ for the projection we may write, thanks to \eqref{eq:galerkin-optim}: 
	\[
	\norm{(tI + A)^{-1} v - U(tI + A)^{-1}  v_\ell}_2 \leq \frac{2}{t + a} \norm{v}_2
	\frac{\max_{z \in [a, b]} |r(z)|}{\min_{z \in (-\infty, 0]} |r(z)|}. 
	\]
	Taking the minimum over the possible choices of the parameters $\lambda_j$ we get \eqref{eq:sim-ris}.
\end{proof}
% 	In \cite{druskin2011adaptive} the authors
% 	suggest to choose the poles with a greedy algorithm in the set obtained mirroring 
% 	the spectrum of $A$ with respect to the imaginary axis. We note that this choice
% 	is in agreement with the poles obtained by Theorem~\ref{thm:lapl-zol}. 

Concerning the rational approximation of the (parameter dependent) exponential, the idea is to rely on its Laplace transform that involves the resolvent:
\begin{equation}\label{eq:laplace-exp}
e^{-tA} = \frac{1}{2\pi \im}\lim_{T \to \infty} \int_{-\im T}^{\im T} e^{st} (sI+A)^{-1}\ ds. 
\end{equation}
In this formulation, it is possible to exploit the Skeleton approximation of $\frac{1}{s+\lambda}$ in order to find a good choice of poles, independently on the parameter $t$. For proving the main result we need the following technical lemma whose proof is given in the Appendix~\ref{app:lapl}.
\begin{lemma} \label{lem:laplinv}
	Let $\mathcal L^{-1}[\widehat r(s)]$ be the inverse Laplace
	transform of $\widehat r(s) = \frac{1}s\frac{p(s)}{p(-s)}$, where
	$p(s)$ is a polynomial of degree $\ell$ with positive 
	real zeros contained in $[a,b]$. Then, 
	\[
	\norm{\mathcal L^{-1}[\widehat r(s)]}_{L^\infty(\mathbb R_+)} \leq \gamma_{\ell,\kappa},\qquad \gamma_{\ell,\kappa}:=2.23 + \frac{2}{\pi} \log\left(
	4\ell \cdot \sqrt{\frac{\kappa}{\pi }}
	\right),
	\]
	where $\kappa=\frac ba$.
\end{lemma}

\begin{theorem}\label{thm:lapl-exp-zol}
	Let $A$ be Hermitian positive definite with spectrum contained in $[a,b]$ and $U$ be an orthonormal basis of $\mathcal{U_R}=\mathcal{RK}_\ell(A,v,\Pole)$,
	where $\Pole=\{\pole_1,\dots,\pole_\ell\} \subseteq [-b,-a]$. 
	Then, $\forall t\in[0,\infty)$, we have the following inequality:
	\begin{equation}
	\label{eq:sim-exp}
	\norm{e^{-tA}v-Ue^{-tA_\ell}v_\ell}_2\leq 4\gamma_{\ell,\kappa}  \norm{v}_2 \max_{z\in[a,b]}|r_\Pole(z)|, 
	\end{equation}
	where  $A_\ell=U^*AU$, $ v_\ell=U^*v$, $\kappa := \frac{b}{a}$, $r_\Pole(z)\in\mathcal R_{\ell,\ell}$ is the rational function  defined by $r_\Pole(z) := \prod_{j = 1}^\ell 
	\frac{z + \pole_j}{z - \pole_j}$ and $\gamma_{\ell,\kappa}$ is the constant defined in Lemma~\ref{lem:laplinv}.
\end{theorem}
\begin{proof}
	We consider the Skeleton approximation of $\frac{1}{s+\lambda}$
	by restricting  the choice 
	of poles in both variables to $ \Pole$
	\[
	f_{\mathrm{skel}}(s,\lambda) := \begin{bmatrix}
	\frac{1}{\lambda - \pole_1} & \ldots &  		   \frac{1}{\lambda - \pole_\ell}
	\end{bmatrix} M^{-1} \begin{bmatrix}
	\frac{1}{s-\pole_1} \\ \vdots \\ \frac{1}{s-\pole_\ell}
	\end{bmatrix},\qquad M_{ij}=-\frac{1}{\pole_i+\pole_j},
	\]
	where $M_{ij}$ denote the entries of $M$.
	Then, by using \eqref{eq:laplace-exp} for $A$ and $A_{\ell}$ we get
	\[
	e^{-tA}v-Ue^{-tA_\ell}v_\ell = \frac{1}{2\pi \im}\lim_{T \to \infty} \int_{-\im T}^{\im T} e^{st} (sI+A)^{-1}v-e^{st} U(sI+A_\ell)^{-1}v_\ell\ ds. 
	\]
	Adding and removing the
	term $e^{st}f_{\mathrm{skel}}(s,A)v = e^{st}Uf_{\mathrm{skel}}(s,A_\ell)U^*v$
	inside the integral (the equality holds
	thanks to Lemma~\ref{lem:exact})
	we obtain the error expression
	\begin{align*}
	e^{-tA} v - U e^{-tA_\ell} v_\ell &= \frac{1}{2\pi \im}
	\lim_{T \to \infty} \int_{-\im T}^{\im T} e^{st} \left[ (sI+A)^{-1} v - U (sI+A_\ell)^{-1}  v_\ell\right]\ ds \\
	&= \frac{1}{2\pi \im} \lim_{T \to \infty} \int_{-\im T}^{\im T} e^{st} (sI+A)^{-1} \left[ I -(sI+A) f_{\mathrm{skel}}(s,A) \right] v\ ds \\
	&- \frac{1}{2\pi \im} \lim_{T \to \infty} U \int_{-\im T}^{\im T} e^{st} (sI+A_\ell)^{-1} \left[ I - (sI+A_\ell)f_{\mathrm{skel}}(s,A_\ell) \right]  v_\ell\ ds\\
	&= \frac{1}{2\pi \im} \lim_{T \to \infty} \int_{-\im T}^{\im T} e^{st} (sI+A)^{-1} r_\Pole(A)r_\Pole(-s)^{-1} v\ ds \\
	&- \frac{1}{2\pi \im} \lim_{T \to \infty} U \int_{-\im T}^{\im T} e^{st} (sI+A_\ell)^{-1} r_\Pole(A_\ell)r_\Pole(-s)^{-1}  v_\ell\ ds.
	\end{align*}
	Since $A$ and $A_\ell$ are normal, the above integrals
	can be controlled by the maximum of the corresponding scalar
	functions on the spectrum of $A$ (and $A_\ell$), which
	yields the bound 
	\[
	\norm{e^{-tA} v - U e^{-tA_\ell}  v_\ell}_2 \leq 
	2 \max_{\lambda \in [a, b]} |h(t,\lambda)|, \qquad 
	h(t,\lambda) := \frac{1}{2\pi \im} \lim_{T \to \infty}  \int_{-iT}^{iT} e^{st} \frac{1}{s+\lambda} \frac{r_\Pole(\lambda)}{r_\Pole(-s)}\ ds. 
	\]
	We note that $r_\Pole(\lambda)$ can be pulled out of the integral, 
	since it does not depend on $s$, and thus the above can be rewritten as 
	\begin{align*}
	h(t,\lambda) = r_\Pole(\lambda) \cdot \mathcal L^{-1}\left[
	\frac{1}{\lambda + s} \frac{p(s)}{p(-s)}
	\right](t) &=r_\Pole(\lambda) \cdot \mathcal L^{-1}\left[\frac{s}{s+\lambda}\right]\star \mathcal L^{-1}\left[
	\frac{1}{s} \frac{p(s)}{p(-s)}
	\right](t)\\
	&=r_\Pole(\lambda) \cdot (\delta(t)-\lambda e^{-\lambda t})\star \mathcal L^{-1}\left[
	\frac{1}{s} \frac{p(s)}{p(-s)}
	\right](t),
	\end{align*}
	where $p(s)$ is as in Lemma~\ref{lem:laplinv} and $\delta(t)$ indicates the Dirac delta function. Since the $1$-norm of $\delta(t)-\lambda e^{-t\lambda}$ is equal to $2$, using Young's inequality we can bound 
	$\norm{h(t,\lambda)}_{\infty} \leq 2 \norm{\mathcal L^{-1}\left[
		\frac{1}{s} \frac{p(s)}{p(-s)}
		\right]}_\infty$. Therefore, we need to estimate the infinity norm of $\mathcal L^{-1}\left[
	\frac{1}{s} \frac{p(s)}{p(-s)}
	\right](t)$.
	Such inverse Laplace transform can be
	uniformly bounded in $t$ by using 	Lemma~\ref{lem:laplinv} with a constant 
	that only depends on $\ell$ and $b/a$:
	\[
	|h(\lambda,  t)| \leq 2\gamma_{\ell,\kappa}|r_\Pole(\lambda)|. 
	\]
	This completes the proof. 
\end{proof}
\begin{remark} \label{rem:laplinvnotoptimal}
	The constant provided by Lemma~\ref{lem:laplinv} is likely 
	not optimal. Indeed, experimentally it seems to hold 
	that 
	$\gamma_{\ell,\kappa} = 1$ for any choice of poles in the 
	negative real axis --- not necessarily contained in $[-b,-a]$ --- and this has been verified in 
	many examples. If this is proved, then the statement
	of Theorem~\ref{thm:lapl-zol} can be made sharper 
	by 
	removing the factor $\gamma_{\ell,\kappa}$. 
\end{remark}

\subsection{Bounds for the rational approximation problems}
Theorem~\ref{thm:lapl-zol} and Theorem~\ref{thm:lapl-exp-zol} show the connection between the error norm and certain rational approximation problems. In this section we discuss the optimal values of such problems in the cases of interests.
\begin{definition} \label{def:theta}
	Let $\Pole\subset \overline{\mathbb C}$ 
	be a finite set, and $I_1, I_2$ closed subsets of $\overline{\mathbb C}$. Then, we define\footnote{We allow the slight abuse of notation of writing 
		$|r(\infty)|$ as the limit of $|r(z)|$ as $|z| \to \infty$, in the case either
		$I_1$ or $I_2$ contains the point at infinity.}
	\[
	\theta_\ell(I_1, I_2, \Pole) :=  \min_{r(z)\in \frac{\mathcal P_\ell}{\Pole}} \frac{\max_{I_1}|r(z)|}{\min_{I_2}|r(z)|}.
	\]
	
\end{definition}

The $\theta_\ell$ functions enjoy some invariance and inclusion properties, 
which we report here, and will be extensively used in the rest of the paper.

\begin{lemma} \label{lem:tp}
	Let $I_1, I_2$ be subsets of the complex plane, and $\Pole\subset\overline{\mathbb C}$. Then, 
	the map $\theta_\ell$ satisfies the following properties:
	\begin{enumerate}[label=(\roman*)]
		\item \label{lem:tp:shift} (shift invariance) For any $t \in \mathbb C$, it holds
		$\theta_\ell(I_1 + t, I_2 + t, \Pole+t) = \theta(I_1, I_2, \Pole)$.
		\item (monotonicity) \label{lem:tp:inclusion} 
		$\theta_\ell(I_1, I_2, \Pole)$ is monotonic with respect to the
		inclusion on the parameters $I_1$ and $I_2$:
		\[
		I_1 \subseteq I_1', I_2 \subseteq I_2' \implies 
		\theta_\ell(I_1, I_2, \Pole) \leq \theta_\ell(I_1', I_2', \Pole). 
		\]
		\item (M\"obius invariance) \label{lem:tp:mobius}
		If $M(z)$ is a M\"obius transform, that is a rational
		function $M(z) = (\alpha z + \beta) / (\gamma z + \delta)$ with $\alpha\delta \neq \beta\gamma$, then 
		\[
		\theta_\ell(I_1, I_2, \Pole) = \theta_\ell(M(I_1), M(I_2), M(\Pole)).
		\]
	\end{enumerate}
\end{lemma}

\begin{proof}
	Property~\ref{lem:tp:shift} follows by \ref{lem:tp:mobius} because applying a shift is a particular M\"obius transformation. Note that, generically,
	when we compose a rational function $r(z) = \frac{p(z)}{h(z)} \in \frac{\mathcal P_\ell}{\Pole}$ 
	with $M^{-1}(z)$ we obtain another 
	rational function of  (at most) the same degree and with poles 
	$M(\Pole)$. Hence, 
	we  obtain 
	\begin{align*}
	\theta_\ell(I_1, I_2, \Pole) &= \min_{r(z)\in \frac{\mathcal P_\ell}{\Pole}} \frac{\max_{I_1}|r(z)|}{\min_{I_2}|r(z)|} =
	\min_{r(z)\in \frac{\mathcal P_\ell}{\Pole}} \frac{\max_{M(I_1)}|r(M^{-1}(z))|}{\min_{M(I_2)}|r(M^{-1}(z))|}\\&=\min_{r(z)\in \frac{\mathcal P_\ell}{M(\Pole)}} \frac{\max_{M(I_1)}|r(z)|}{\min_{M(I_2)}|r(z)|} = \theta_\ell(M(I_1), M(I_2), M(\Pole)). 
	\end{align*}
	Property~\ref{lem:tp:inclusion} follows easily from the fact that the maximum taken on a larger set is larger, 
	and the minimum taken on a larger set is smaller. 
\end{proof}
Now, we  consider the related optimization problem, obtained by allowing
$\Pole$ to vary:
\begin{equation}\label{eq:zol3}
\min_{\substack{\Pole\subset\overline{\mathbb  C},  |\Pole|=\ell} }\theta_\ell(I_1,I_2,\Pole) = \min_{r(z) \in \mathcal R_{\ell,\ell}} \frac{\max_{z \in I_1} |r(z)|}{\min_{z \in I_2} |r(z)|}.
\end{equation}
The latter was posed and studied by Zolotarev
in 1877 \cite{zolotarev1877application}, and it is commonly known as
the \emph{third Zolotarev problem}. We refer to \cite{Beckermann2011}
for a  modern reference where the theory is used to recover bounds on
the convergence of rational Krylov methods and ADI iterations for solving
Sylvester equations. 

In the case $I_1=-I_2=[a,b]$ \eqref{eq:zol3} simplifies to 
\[\min_{r(z) \in \mathcal R_{\ell,\ell}} \frac{\max_{z \in [a,b]} |r(z)|}{\min_{z \in [a,b]} |r(-z)|}\]
which admits the following explicit estimate. 
\begin{theorem}[Zolotarev] \label{thm:zolotarev}
	Let $I = [a, b]$, with $0<a<b$. Then 
	\[
	\min_{\Pole\subset\overline{\mathbb  C}, \ |\Pole|=\ell }\theta_\ell(I,-I,\Pole) \leq 4\rho_{[a,b]}^{\ell},\qquad \rho_{[a,b]}:=\exp\left(-\frac{\pi^2}{\log\left(4 \kappa\right)}\right),\qquad \kappa=\frac ba.
	\]
	In addition, the optimal rational function $r^{[a,b]}_\ell(z)$ that realizes the minimum
	has the form
	\[
	r^{[a,b]}_\ell(z) := \frac{p^{[a,b]}_\ell(z)}{p^{[a,b]}_\ell(-z)}, \qquad 
	p^{[a,b]}_\ell(z) := \prod_{j = 1}^\ell (z + \pole^{[a,b]}_{j,\ell}), \qquad 
	\pole^{[a,b]}_{j,\ell} \in -I. 
	\]
	We denote by $\Pole_{\ell}^{[a,b]} := \{ \pole_{1,\ell}^{[a,b]}, \dots, \pole_{\ell,\ell}^{[a,b]} \}$ the set of poles of $r^{[a,b]}_\ell(z)$.
\end{theorem}
Explicit expression for the elements of $\Pole_\ell^{[a,b]}$ are available in terms of elliptic functions, see \cite[Theorem 4.2]{druskin2009}.

\begin{remark}
	The original version of Zolotarev's result involves
	$\mathrm{exp}(-\frac{\pi^2}{\mu(\kappa^{-1})})$
	in place of $\rho_{[a, b]}$, 
	where $\mu(\cdot)$ is the Gr\"otzsch ring function. For simplicity, 
	in this paper 
	we prefer the slightly suboptimal form involving the logarithm. 
	We remark that for large $\kappa$ (which is usually the case when 
	considering rational Krylov methods) the difference is negligible
	\cite[Section~17.3]{abramowitz1965handbook}. 
\end{remark}

We use Theorem~\ref{thm:zolotarev} and the M\"obius invariance property as building blocks for bounding \eqref{eq:sim-ris}.
The idea is to map the set $[-\infty, 0] \cup [a, b]$ into $[-1, -\widehat a]\cup[\widehat a, 1]$ --- for some $\widehat a\in (0,1)$ --- with a M\"obius transformation; then  make use of Theorem~\ref{thm:zolotarev} and Lemma~\ref{lem:tp}-\ref{lem:tp:mobius} to provide
a convenient choice of $\Pole$ for the original problem. 

\begin{lemma} \label{lem:mobius1d}
	The M\"obius transformation \[
	T_C(z):=\frac{\Delta + z - b}{\Delta - z + b},  \qquad 
	% \Delta := \sqrt{b^2 - a-\frac{3a^2}{4}},
	\Delta := \sqrt{b^2 - ab},
	\]
	maps $[-\infty,0]\cup[a,b]$ into $[-1, -\widehat a]\cup[\widehat a, 1]$, with $\widehat a := \frac{\Delta + a - b}{\Delta - a + b}=\frac{b-\Delta }{\Delta +b}$. The inverse map
	$T_C(z)^{-1}$ is given by:
	\[
	T_C^{-1}(z):= \frac{(b+\Delta)z+b-\Delta}{1+z}.
	\]
	Moreover, for any $0<a<b$ it holds $\widehat a^{-1} \leq \frac{4b}{a}$, 
	and therefore 
	$\rho_{[\widehat a,1]}\leq \rho_{[a,4b]}$.
\end{lemma}
\begin{proof}
	By direct substitution, we have $T_C(-\infty) = -1$, and 
	$T_C(b) = 1$; moreover, again by direct computation one verifies that 
	$T_C(0) + T_C(a) = 0$, which implies that $T_C([-\infty, 0]) = [-1, -\widehat a]$ and $T_C([a, b]) = [\widehat a, 1]$. Then, we have 
	\[
	\widehat a^{-1} = \frac{\Delta + b}{b - \Delta}, \qquad 
	\Delta = b \sqrt{1 - a/b}.
	\]
	Using the relation $\sqrt{1 - t} \leq 1 - \frac{t}{2}$ for any $0 \leq t \leq 1$, we obtain that 
	$
	\widehat a^{-1} \leq \frac{2b - \frac{a}{2}}{\frac{a}{2}} \leq 4\frac{b}{a}, 
	$
	which concludes the proof. 
\end{proof}

\begin{remark} \label{rem:estimate-cauchy-exponent}
	We note that the estimate $\rho_{[\widehat a, 1]} \leq \rho_{[a,4b]}$ 
	is asymptotically tight, that is the limit of 
	$\rho_{[\widehat a, 1]} / \rho_{[a,4b]} \to 1$ as $b/a \to \infty$. 
	For instance, if $b/a = 10$ then the relative error between the two
	quantities is about $2\cdot 10^{-2}$, and for $b/a = 1000$ 
	about $5\cdot 10^{-5}$. Since the interest for this approach is 
	in dealing with matrices that are not well-conditioned, we consider the
	error negligible in practice. 
\end{remark}

In light of Theorem~\ref{thm:zolotarev} and Lemma~\ref{lem:mobius1d}, we consider  the choice 
\begin{equation} \label{eq:cauchy1d-poles}
\Pole_{C,\ell}^{[a,b]} := T_C^{-1}(\Pole_\ell^{[\widehat a,1]})
\end{equation} in  Theorem~\ref{thm:lapl-zol}.
This yields the following estimate.
\begin{corollary}\label{cor:cauchy}
	Let $A$ be Hermitian positive definite with spectrum contained in $[a,b]$ and $U$ be an orthonormal basis of $\mathcal{U_R}=\mathcal{RK}_\ell(A,v,\Pole_{C,\ell}^{[a,b]})$. Then, $\forall t\in[0,\infty)$
	\begin{equation}
	\norm{(tI+A)^{-1}v-U(tI+A_\ell)^{-1}v_\ell}_2\leq \frac{8}{t+a}  \norm{v}_2 \rho_{[a,4b]}^\ell,
	\end{equation}
	where  $A_\ell=U^*AU$ and $v_\ell=U^*v$.
\end{corollary}

When considering Laplace-Stieltjes functions, we may choose as 
poles $\Pole_\ell^{[a,b]}$ which are 
the optimal Zolotarev poles on the interval $[a, b]$. This 
enables to prove the following result, which builds on
Theorem~\ref{thm:lapl-exp-zol}. 

\begin{corollary} \label{cor:lapl}
	Let $A$ be Hermitian positive definite with spectrum contained in $[a,b]$ and $U$ be an orthonormal basis of  $\ \mathcal{U_R}=\mathcal{RK}_\ell(A,v,\Pole_{\ell}^{[a,b]})$.
	Then, $\forall t\in[0,\infty)$
	\begin{equation}\label{eq:13}
	\norm{e^{-tA}v-Ue^{-tA_\ell}v_\ell}_2\leq 8 \gamma_{\ell,\kappa}
	\norm{v}_2 \rho_{[a,b]}^{\frac \ell2},
	\end{equation}
	where  $A_\ell=U^*AU$  and $v_\ell=U^*v$.
\end{corollary} 
\begin{proof}
	The proof relies on the fact that the optimal 
	Zolotarev function evaluated on the interval $[a, b]$ can be
	bounded by $2 \rho_{[a,b]}^{\frac \ell2}$ \cite[Theorem 3.3]{Beckermann2019}. Since
	its zeros and poles are symmetric with respect to the
	imaginary axis and real, we can apply Theorem~\ref{thm:lapl-exp-zol} to obtain \eqref{eq:13}. 
\end{proof}
\subsection{Convergence bounds for Stieltjes functions}\label{sec:conv1d}Let us consider $f(z)$ a Stieltjes function of the general form \eqref{eq:general-stielt}. Then the error of the rational Krylov method for approximating $f(A)v$ is given by
\begin{align*}
\norm{f(A)v - Uf(A_\ell)v_\ell}_2 &= \norml{\int_0^\infty 
	\left[ g(t,A)v - Ug(t,A_{\ell}) v_\ell \right]\mu(t)  \ dt}_2\\
&\leq\int_0^\infty 
\norml{ g(t,A)v - Ug(t,A_{\ell}) v_\ell }_2\mu(t)  \ dt
\end{align*}
where $g(t,A)$ is either a parameter dependent 
exponential or resolvent function. Therefore 
Corollary~\ref{cor:cauchy} and Corollary~\ref{cor:lapl} provide all the ingredients to study the error of the rational Krylov projection, when the suggested pole selection strategy is adopted.
\begin{corollary}\label{cor:lapl-stielt}
	Let $f(z)$ be a Laplace-Stieltjes function, $A$ be Hermitian positive definite with spectrum contained in $[a,b]$, $U$ be an orthonormal basis of $\mathcal{U_R}=\mathcal{RK}_\ell(A,v,\Pole_{\ell}^{[a,b]})$ and $x_\ell=Uf(A_\ell)v_\ell$ with $A_\ell=U^*AU$ and $v_\ell=U^*v$. Then
	\begin{equation}
	\norm{f(A)v-x_\ell}_2\leq 8 \gamma_{\ell,\kappa} f(0^+)  \norm{v}_2 \rho_{[a,b]}^{\frac \ell2},
	\end{equation}
	where $\gamma_{\ell,\kappa}$ is defined as in Theorem~\ref{thm:lapl-exp-zol}, and $f(0^+) := \lim_{z \to 0^+} f(z)$. 
\end{corollary}
\begin{proof}
	Since $f(z)$ is a Laplace-Stieltjes function, we can express the error as follows:
	\begin{align*}
	\norm{f(A)v - x_\ell}_2 &\leq \int_0^\infty 
	\norml{ e^{-tA}v - Ue^{-tA_{\ell}} U^* v }_2\mu(t)  \ dt\\
	&\leq 8\gamma_{\ell,\kappa}\int_0^\infty 
	\mu(t)   \ dt \norm{v}_2\rho_{[a,b]}^{\frac \ell 2}\\ &=8 \gamma_{\ell,\kappa}f(0^+)  \norm{v}_2 \rho_{[a,b]}^{\frac \ell2},
	\end{align*}
	where we applied \eqref{eq:galerkin-optim} and Corollary~\ref{cor:lapl}
	to obtain the second inequality.
\end{proof}

\begin{remark} \label{rem:shift-bound}
	In order to be meaningful, 
	Corollary~\ref{cor:lapl-stielt} requires the function $f(z)$ to be
	finite over $[0, \infty)$, which might not be the case in general (consider 
	for instance $x^{-\alpha}$, which is both Cauchy and Laplace-Stieltjes). 
	Nevertheless, the result can be applied to $f(z + \eta)$, which
	is always completely monotonic for a positive $\eta$,
	by taking $0 < \eta < a$. A value of $\eta$ 
	closer to $a$ gives a slower decay rate, but a smaller constant $f(0^+)$. Similarly,
	if $f(z)$ happens to be completely monotonic on an 
	interval larger than $[0, \infty)$, then bounds with a faster asymptotic
	convergence rate (but a larger constant)
	can be obtained considering $\eta < 0$. 
\end{remark}

Corollary~\ref{cor:cauchy} allows to state the corresponding bound for Cauchy-Stieltjes functions.
The proof is analogous to the one of Corollary~\ref{cor:lapl-stielt}. 

\begin{corollary} \label{cor:cauchy-stielt}
	Let $f(z)$ be a Cauchy-Stieltjes function, $A$ be Hermitian positive definite with spectrum contained in $[a,b]$, $U$ be an orthonormal basis of $\mathcal{U_R}=\mathcal{RK}_\ell(A,v,\Pole_{C,\ell}^{[a,b]})$
	with $\Pole_{C,\ell}^{[a,b]}$ as in \eqref{eq:cauchy1d-poles}
	and $x_\ell=Uf(A_\ell)v_\ell$ with $A_\ell=U^*AU$ and $v_\ell=U^*v$. Then
	\begin{equation}
	\norm{f(A)v-x_G}_2\leq 8f(a)  \norm{v}_2 \rho_{[a,4b]}^\ell.
	\end{equation}
\end{corollary}

\subsection{Nested sequences of poles} \label{sec:nested}

From the computational perspective, 
it is more convenient to have a nested sequence of subspaces 
$\mathcal W_{1} \subseteq \ldots 
\mathcal{W}_j \subseteq \mathcal W_{j+1} \subseteq \ldots$, so
that $\ell$ can be chosen adaptively. For example, in \cite{guettel2013blackbox} the authors propose a  greedy
algorithm for the selection of the poles taylored to the evaluation of Cauchy-Stieltjes matrix functions. See \cite{druskin2010adaptive,druskin2011adaptive} for greedy pole selection strategies to be applied in different --- although closely related --- contexts.

The choices of poles proposed in the previous sections require to 
a priori determine the degree $\ell$ of the approximant $x_\ell$. Given a target accuracy,  one can use the convergence bounds in Corollary~\ref{cor:lapl-stielt}--\ref{cor:cauchy-stielt} to determine $\ell$. Unfortunately, this is likely to overestimate the minimum value of $\ell$ that provides the sought accuracy.

An option, that allows to overcome this limitation, is to rely on the method of \emph{equidistributed sequences} (EDS), as described in \cite[Section 4]{druskin2009}. The latter exploits the limit --- as $\ell\to\infty$ --- of the measures generated by the set of points $\Psi_\ell^{[a,b]},\Psi_{C,\ell}^{[a,b]}$  to return infinite sequences of poles that are guaranteed to provide the same asymptotic rate of convergence. More specifically, the EDS $\{\widetilde \sigma_j\}_{j\in\mathbb N}$ associated with $\Psi_{\ell}^{[a,1]}$ is obtained with the following steps:
\begin{enumerate}[label=(\roman*)]
	\item Select $\zeta\in\mathbb R^+\setminus\mathbb Q$ and generate the sequence $\{s_j\}_{j\in\mathbb N}:=\{0,\ \zeta -\lfloor\zeta\rfloor,\ 2\zeta -\lfloor2\zeta\rfloor,\ 3\zeta -\lfloor3\zeta\rfloor,\ \dots\}$, where $\lfloor\cdot\rfloor$ indicates the greatest integer less than or equal to the argument; 
	this sequence has as asymptotic distribution (in the sense of EDS) 
	the Lebesgue measure
	on $[0, 1]$. 
	\item Compute the sequence $\{t_j\}_{j\in\mathbb N}$ such that $g(t_j)=s_j$ where
	$$
	g(t):=\frac{1}{2M}\int_{a^2}^t\frac{dy}{\sqrt{(y-a^2)y(1-y)}},\qquad M:=\int_0^1\frac{dy}{\sqrt{(1-y^2)(1-(1-a^2)y^2)}},
	$$
	\item Define $\widetilde \sigma_j:=\sqrt{t_j}$.
\end{enumerate}
More generally, the EDS associated with $\Psi_\ell^{[a,b]},\Psi_{C,\ell}^{[a,b]}$ are obtained by applying either a scaling or the M\"obius transformation \eqref{eq:cauchy1d-poles} to the EDS for $\Psi_{\ell}^{[a,1]}$.

In our implementation, only the finite portion $\{\widetilde \sigma_j\}_{j=0,\dots, \ell-1}$ is --- incrementally --- generated for computing $x_\ell$. As starting irrational number we select $\zeta=\frac{1}{\sqrt 2}$ and each quantity $t_j$ is approximated by applying the Newton's method to the equation $g(t_j)-s_j=0$. The initialization of
the Newton iteration is done by approximating $\hat t \mapsto g(e^{\hat t}) - s_j$ 
with a linear function on the domain of interest, and then using the
exponential of its only root as starting point. This is done beforehand selecting $t = a^2$ and $t = a$ 
as interpolation points; in our experience, with such starting point Newton's
method converges in a few steps. 
\subsection{Some numerical tests}
\subsubsection{Laplace-Stieltjes functions}\label{ex:lapl1d}
Let us consider the 1D diffusion problem over $[0, 1]$ with
zero Dirichlet boundary conditions
\[
\frac{\partial u}{\partial t} = \epsilon \frac{\partial^2 u}{\partial x^2} + f(x), 
\qquad u(x, 0) \equiv 0, \qquad \epsilon = 10^{-2},
\]
discretized using central finite differences in space with $50000$ points, 
and integrated 
by means of the exponential Euler method with time step $\Delta t = 0.1$. This requires
to evaluate the action of the Laplace-Stieltjes matrix function $\varphi_1(\frac{\epsilon}{h^2} \Delta t A) v$, where $A$ is the
tridiagonal matrix $A = \mathrm{tridiag}(-1, 2, -1)$. 
We test the convergence rates of various choices of poles by measuring the absolute error when using a random vector $v$. Figure~\ref{fig:exp-lapl-1d}-(left) reports the results associated with:
the poles from Corollary~\ref{cor:lapl}, the corresponding EDS computed as described in Section~\ref{sec:nested} and  the extended Krylov method. 
It is visible that the three approximations have the same convergence rate, although the choice of poles from Corollary~\ref{cor:lapl} and the EDS performs slightly better. We mention that, since $A$ is ill-conditioned, polynomial Krylov performs 
poorly on this example. 

We keep the same settings and we test the convergence rates for the Laplace-Stieltjes function $z^{-\frac 32}W(z)$ where $W(z)$ is the Lambert W function \cite{kalugin2012stieltjes}. The plot in Figure~\ref{fig:exp-lapl-1d}-(right) shows that after about $10$ iterations the convergence rate of the extended Krylov method deteriorates, while the poles from Corollary~\ref{cor:lapl} and the EDS provide the best convergence rate.
\begin{figure}\centering
	\begin{tikzpicture}
	\begin{semilogyaxis}[
	title = {$f(z) = \varphi_1(z)$},
	ytick = {1e-11, 1e-8, 1e-5, 1e-2,1e1},
	legend style = {at={(1.05,1)}, anchor = north west}, 
	xlabel = Iterations ($\ell$), ymin = 1e-12, ymax = 1e2,
	ylabel = {$\norm{x  -x_{\ell}}_2$}, width = .45\linewidth, 
	height = .25\textheight]			
	\addplot[blue, mark=*] table[x index = 0, y index = 1] {1d-50000-phi.dat};
	\addplot[red, mark=square*] table[x index = 0, y index = 2] {1d-50000-phi.dat};
	\addplot[green, mark=diamond*] table[x index = 0, y index = 3] {1d-50000-phi.dat};				
	\addplot[red, dashed] table[x index = 0, y index = 4] {1d-50000-phi.dat};
	%\legend{Extended Krylov, Poles from Cor.~\ref{cor:cauchy}, \texttt{markovfunm}, Adaptive strategy}
	\end{semilogyaxis}
	\end{tikzpicture}~\begin{tikzpicture}
	\begin{semilogyaxis}[
	title = {$f(z) = z^{-\frac 32} W(z)$},
	legend style = {at={(1.05,1)}, anchor = north west}, 
	xlabel = Iterations ($\ell$), 
	ytick = \empty,ymin = 1e-12, ymax = 1e2,
	width = .45\linewidth, 
	height = .25\textheight]			
	\addplot[blue, mark=*] table[x index = 0, y index = 1] {1d-50000-lwz.dat};
	\addplot[red, mark=square*] table[x index = 0, y index = 2] {1d-50000-lwz.dat};
	\addplot[green, mark=diamond*] table[x index = 0, y index = 3] {1d-50000-lwz.dat};	
	\addplot[domain=1:50,red, dashed]{14.16 * 8*(2.23 + 2/3.14 * ln(4*x*sqrt(1e9/3.14))) * sqrt(exp(-3.14^2/(ln(4*1e9))))^x};		
	%\addplot[red, dashed] table[x index = 0, y index = 4] {1d-50000-lwz.dat};
	\legend{EK, Cor.~\ref{cor:lapl}, EDS, Bound};
	\end{semilogyaxis}
	\end{tikzpicture}
	\caption{Convergence history of the different projection spaces for the evaluation of 
		$\varphi_1(A)v$ and $A^{-\frac 32}W(A)v$ for a matrix argument of size $50000 \times 50000$. 
		The methods tested are extended Krylov (EK), rational Krylov with the poles from 
		Corollary~\ref{cor:lapl} and rational Krylov with nested poles obtained as in Section~\ref{sec:nested}. The bound in the left figure is obtained directly from Corollary~\ref{cor:lapl}. The bound in the right figure has been obtained as in Remark~\ref{rem:shift-bound}. 
	}
	\label{fig:exp-lapl-1d}
\end{figure}
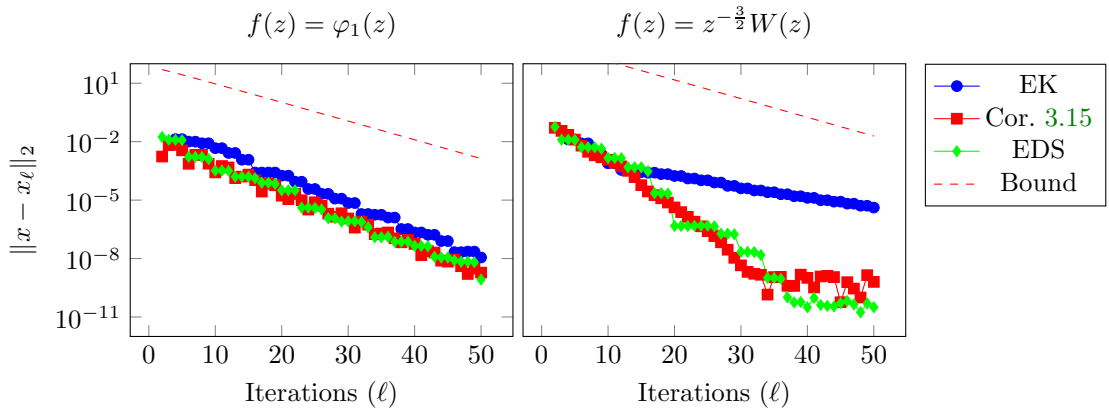
\subsubsection{Inverse square root}\label{ex:cauchy1d}
Let us test the pole selection strategies for Cauchy-Stieltjes functions, by  considering
the evaluation of $f(z) = z^{-\frac 12}$ on the $n\times n$ matrix $\mathrm{tridiag}(-1, 2, -1)$, for $n=10^4,5\cdot 10^4, 10^5$. The list of methods that we consider includes:  the poles $\Pole_{C,\ell}^{[a,b]}$ from Corollary~\ref{cor:cauchy}, the associated EDS, the extended Krylov method and the adaptive strategy
proposed in \cite{guettel2013blackbox} for Cauchy-Stieltjes functions. The latter is implemented in the
\texttt{markovfunmv} package available at \url{http://guettel.com/markovfunmv/} which we used for producing  the  results 
reported in Figure~\ref{fig:exp-cauchy-1d}. 
The poles  from Corollary~\ref{cor:cauchy} and the extended Krylov method yield the best and the worst convergence history, respectively, for all values of $n$. The EDS and \texttt{markovfunm} perform similarly for $n=10^4$, but as $n$ increases the decay rate of \texttt{markovfunm} deteriorates significantly. 

We consider a second numerical experiment which keeps the same settings apart from the size of the matrix argument which is fixed to $n=10^5$. Then, we measure the number of iterations and the computational time needed by the methods using nested sequences of poles, i.e. EK, EDS, \texttt{markovfunm}, to reach different target values for the relative error $\frac{\norm{x-x_{\ell}}_2}{\norm{x}_2}$. The EK method has the cheapest iteration cost because it exploits the pre-computation of the Cholesky factor of the matrix $A$ for the computation of the orthogonal basis. However, as testified by the results in Table~\ref{tab:times}, the high number of iterations makes EK competitive only for the high relative error $10^{-1}$. The iteration costs of EDS and \texttt{markovfunm} is essentially the same since they only differ in the computation of the poles, which requires a negligible portion of the computational time. Therefore, the comparison between EDS and \texttt{markovfunm} goes along with the number of iterations which makes the former more efficient\footnote{To make a fair comparison between the methods, for this test we relied 
	on the rational Arnoldi implementation found in \texttt{markovfunm}
	for the implementation of Algorithm~\ref{alg:subspace} using
	EDS poles.}. We remark that in the situation where precomputing the
Cholesky gives a larger computational benefit, and memory is not an
issue, EK may be competitive
again. 

We conclude the numerical experiments on the inverse square root by considering matrix arguments with different distributions of the eigenvalues. More precisely, we set $A$ as the diagonal matrix of dimension $n=5\cdot 10^4$ with the following spectrum configurations:
\begin{itemize}
	\item[$(i)$] Equispaced values in the interval $[\frac 1n, 1]$,
	\item[$(ii)$] Eigenvalues of $\mathrm{trid}(-1,2+10^{-3},-1)$ (shifted Laplacian),
	\item[$(iii)$] $20$ Chebyshev points in $[10^{-3}, 10^{-1}]$ and $n-20$ Chebyshev points in $[10, 10^3]$.
\end{itemize}
The convergence histories of the different projection spaces are reported in Figure~\ref{fig:exp-cauchy-1d-eigenvalues}. For all the eigenvalues configurations, EDS and \texttt{markovfunm} provide comparable performances. The poles from Corollary~\ref{cor:cauchy} performs as EDS and \texttt{markovfunm} on $(ii)$ and slightly better on $(i)$ and $(iii)$. Once again, EK is the one providing the slowest convergence rate on all examples.  
\begin{figure} 
	\begin{tikzpicture}
	\begin{semilogyaxis}[
	title = {$f(z) = z^{-\frac 12}$, $n = 10000$},
	legend style = {at={(1.05,1)}, anchor = north west}, 
	xlabel = Iterations ($\ell$), 
	ylabel = {$\norm{x  -x_{\ell}}_2$}, width = .48\linewidth, 
	height = .25\textheight, ytick = {1e-11, 1e-7, 1e-3, 1e1, 1e5}]			
	\addplot[blue, mark=*] table[x index = 0, y index = 1] {1d-10000-invsqrt.dat};
	\addplot[red, mark=square*] table[x index = 0, y index = 2] {1d-10000-invsqrt.dat};
	\addplot[green, mark=diamond*] table[x index = 0, y index = 3] {1d-10000-invsqrt.dat};
	\addplot[black, mark=star] table[x index = 0, y index = 4] {1d-10000-invsqrt.dat};		
	\addplot[red, dashed] table[x index = 0, y index = 5] {1d-10000-invsqrt.dat};
	%\legend{Extended Krylov, Poles from Cor.~\ref{cor:cauchy}, \texttt{markovfunm}, Adaptive strategy}
	\end{semilogyaxis}
	\end{tikzpicture}~\begin{tikzpicture}
	\begin{semilogyaxis}[
	title = {$f(z) = z^{-\frac 12}$, $n = 50000$},
	legend style = {at={(1.05,1)}, anchor = north west}, 
	xlabel = Iterations ($\ell$), 
	ylabel = {$\norm{x  -x_{\ell}}_2$}, width = .48\linewidth, 
	height = .25\textheight, ytick = {1e-9, 1e-6, 1e-3, 1e0,1e3}, ymax=5e4]			
	\addplot[blue, mark=*] table[x index = 0, y index = 1] {1d-50000-invsqrt.dat};
	\addplot[red, mark=square*] table[x index = 0, y index = 2] {1d-50000-invsqrt.dat};
	\addplot[green, mark=diamond*] table[x index = 0, y index = 3] {1d-50000-invsqrt.dat};
	\addplot[black, mark=star] table[x index = 0, y index = 4] {1d-50000-invsqrt.dat};				
	\addplot[red, dashed] table[x index = 0, y index = 5] {1d-10000-invsqrt.dat};
	%\legend{Extended Krylov, Poles from Cor.~\ref{cor:cauchy}, \texttt{markovfunm}, Adaptive strategy}
	\end{semilogyaxis}
	\end{tikzpicture} \\
	\begin{tikzpicture}
	\begin{semilogyaxis}[
	title = {$f(z) = z^{-\frac 12}$, $n = 100000$},
	legend style = {at={(1.5,.75)}, anchor = north west}, 
	xlabel = Iterations ($\ell$), 
	ylabel = {$\norm{x  -x_{\ell}}_2$}, width = .48\linewidth, 
	height = .25\textheight, ytick = {1e-9, 1e-6, 1e-3, 1e0, 1e3}, ymax=5e4]			
	\addplot[blue, mark=*] table[x index = 0, y index = 1] {1d-100000-invsqrt.dat};
	\addplot[red, mark=square*] table[x index = 0, y index = 2] {1d-100000-invsqrt.dat};
	\addplot[green, mark=diamond*] table[x index = 0, y index = 3] {1d-100000-invsqrt.dat};
	\addplot[black, mark=star] table[x index = 0, y index = 4] {1d-100000-invsqrt.dat};				
	\addplot[red, dashed] table[x index = 0, y index = 5] {1d-10000-invsqrt.dat};
	\legend{EK, Cor.~\ref{cor:cauchy}, EDS,  \texttt{markovfunm}, Bound}
	\end{semilogyaxis}
	\end{tikzpicture}
	\caption{Convergence history of the different projection spaces for the evaluation of $A^{-\frac 12}v$, with $A=\mathrm{trid}(-1, 2, -1)$,  
		for different sizes $n$ of the matrix argument. The methods tested are extended Krylov (EK), rational Krylov with the poles from 
		Corollary \ref{cor:cauchy}, rational Krylov with nested poles obtained as in Section~\ref{sec:nested} (EDS) and rational Krylov with the poles of \texttt{markovfunm}. The bound is obtained from Corollary~\ref{cor:cauchy}.}
	\label{fig:exp-cauchy-1d}
\end{figure}
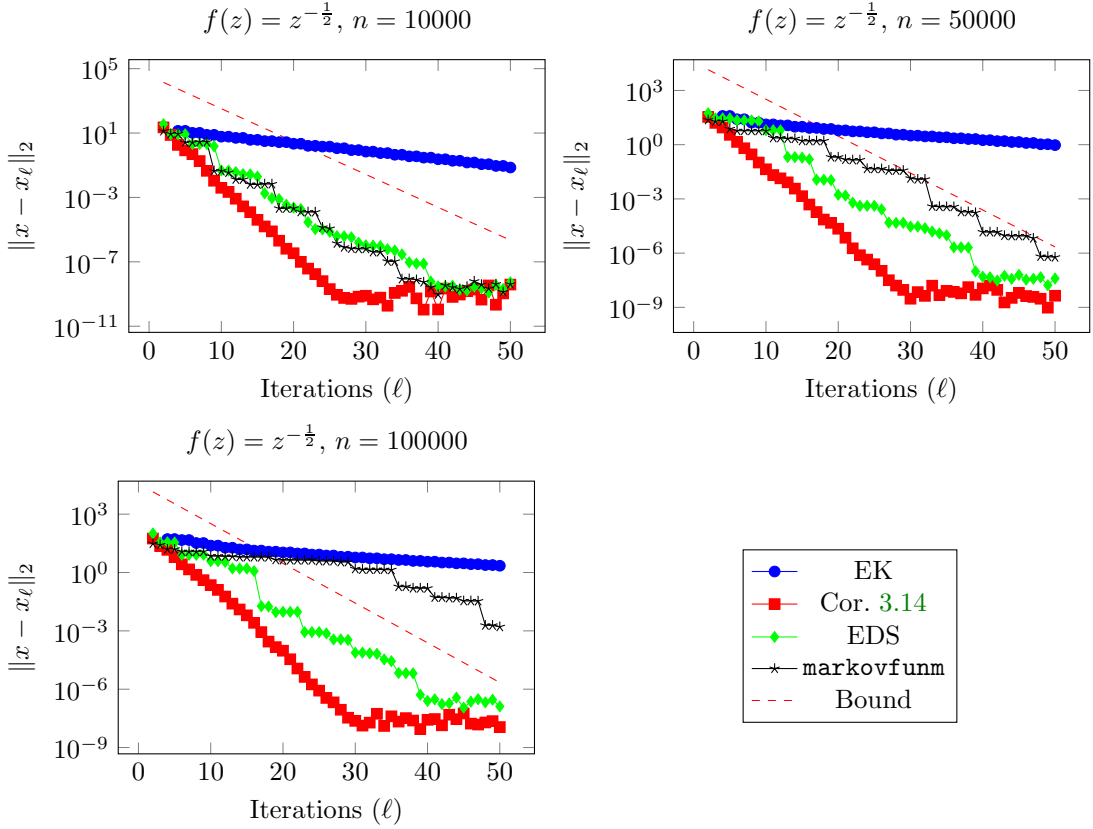

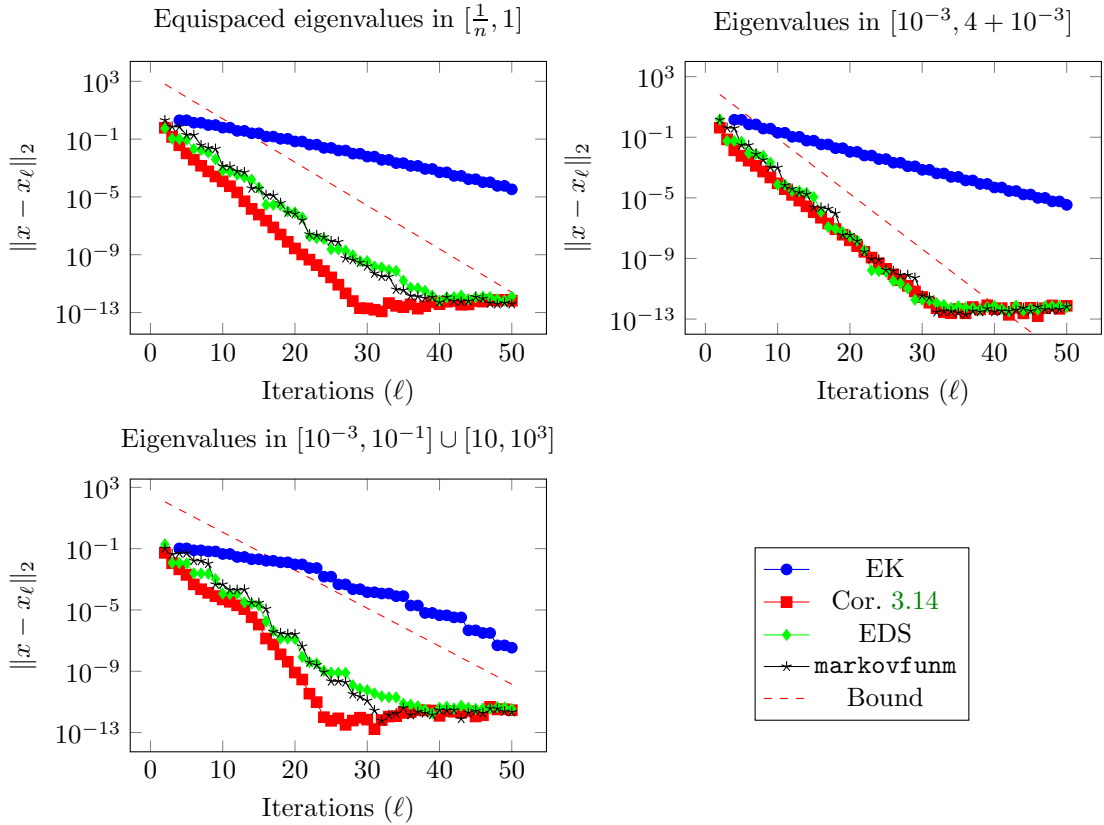
\begin{figure} 
	\begin{tikzpicture}
	\begin{semilogyaxis}[
	title = {Equispaced eigenvalues in $[\frac 1n, 1]$},
	legend style = {at={(1.05,1)}, anchor = north west}, 
	xlabel = Iterations ($\ell$), 
	ylabel = {$\norm{x  -x_{\ell}}_2$}, width = .48\linewidth, 
	height = .25\textheight, ytick = {1e-13,1e-9, 1e-5, 1e-1, 1e3}]			
	\addplot[blue, mark=*] table[x index = 0, y index = 1] {1d-50000-invsqrt-uniform.dat};
	\addplot[red, mark=square*] table[x index = 0, y index = 2] {1d-50000-invsqrt-uniform.dat};
	\addplot[green, mark=diamond*] table[x index = 0, y index = 3] {1d-50000-invsqrt-uniform.dat};
	\addplot[black, mark=star] table[x index = 0, y index = 4] {1d-50000-invsqrt-uniform.dat};				
	\addplot[red, dashed] table[x index = 0, y index = 5] {1d-50000-invsqrt-uniform.dat};
	%\legend{Extended Krylov, Poles from Cor.~\ref{cor:cauchy}, \texttt{markovfunm}, Adaptive strategy}
	\end{semilogyaxis}
	\end{tikzpicture}~\begin{tikzpicture}
	\begin{semilogyaxis}[
	title = {Eigenvalues in $[10^{-3}, 4 + 10^{-3}]$},
	legend style = {at={(1.05,1)}, anchor = north west}, 
	xlabel = Iterations ($\ell$), 
	ylabel = {$\norm{x  -x_{\ell}}_2$}, width = .48\linewidth, 
	height = .25\textheight,ytick = {1e-13,1e-9, 1e-5, 1e-1, 1e3}, ymin=1e-14, ymax=1e4]			
	\addplot[blue, mark=*] table[x index = 0, y index = 1] {1d-50000-invsqrt-well.dat};
	\addplot[red, mark=square*] table[x index = 0, y index = 2] {1d-50000-invsqrt-well.dat};
	\addplot[green, mark=diamond*] table[x index = 0, y index = 3] {1d-50000-invsqrt-well.dat};
	\addplot[black, mark=star] table[x index = 0, y index = 4] {1d-50000-invsqrt-well.dat};				
	\addplot[red, dashed] table[x index = 0, y index = 5] {1d-50000-invsqrt-well.dat};
	%\legend{Extended Krylov, Poles from Cor.~\ref{cor:cauchy}, \texttt{markovfunm}, Adaptive strategy}
	\end{semilogyaxis}
	\end{tikzpicture} \\
	\begin{tikzpicture}
	\begin{semilogyaxis}[	
	title = {Eigenvalues in $[10^{-3}, 10^{-1}] \cup [10, 10^3]$},
	legend style = {at={(1.5,.75)}, anchor = north west}, 
	xlabel = Iterations ($\ell$), 
	ylabel = {$\norm{x  -x_{\ell}}_2$}, width = .48\linewidth, 
	height = .25\textheight,,ytick = {1e-13,1e-9, 1e-5, 1e-1, 1e3}]			
	\addplot[blue, mark=*] table[x index = 0, y index = 1] {1d-50000-invsqrt-gap.dat};
	\addplot[red, mark=square*] table[x index = 0, y index = 2] {1d-50000-invsqrt-gap.dat};
	\addplot[green, mark=diamond*] table[x index = 0, y index = 3] {1d-50000-invsqrt-gap.dat};
	\addplot[black, mark=star] table[x index = 0, y index = 4] {1d-50000-invsqrt-gap.dat};				
	\addplot[red, dashed] table[x index = 0, y index = 5] {1d-50000-invsqrt-gap.dat};
	\legend{EK, Cor.~\ref{cor:cauchy}, EDS,  \texttt{markovfunm}, Bound}
	\end{semilogyaxis}
	\end{tikzpicture}
	\caption{Convergence history of the different projection spaces for the evaluation of $A^{-\frac 12}v$ for a diagonal matrix argument of size $50000 \times 50000$ with different
		eigenvalue distributions. The methods tested are extended Krylov (EK), rational Krylov with the poles from 
		Corollary \ref{cor:cauchy}, rational Krylov with nested poles obtained as in Section~\ref{sec:nested} (EDS) and rational Krylov with the poles of \texttt{markovfunm}. The bound is obtained from Corollary~\ref{cor:cauchy}.}
	\label{fig:exp-cauchy-1d-eigenvalues}
\end{figure}

\begin{table}
	\centering 
	\caption{Comparison of the time and number of iterations required for computing $A^{-\frac 12}v$ 
		with different relative tolerances using \texttt{markovfunm}, EDS, and extended Krylov. The argument  $A$ is the $100000 \times 100000$ matrix
		$\mathrm{trid}(-1, 2, -1)$.} \label{tab:times}
	\vskip 10pt
	\begin{tabular}{c|cc|cc|cc}
		\toprule
		&\multicolumn{2}{c|}{\texttt{markovfunm}}
		&\multicolumn{2}{c|}{EDS}
		&\multicolumn{2}{c}{extended Krylov}\\
		Relative error & Time (s) & Its & Time (s) & Its & Time (s) & Its \\
		\midrule
		$10^{-1}$   &	$0.37$ & $18$ &	$0.10$ & $7$  &	$0.32$ &	$20$  \\
		$10^{-2}$  &	$0.76$	& $29$ &	$0.26$	& $14$ &	$2.17$ &	$64$  \\
		$10^{-3}$ &	$1.17$	& $38$ &	$0.37$	& $18$ &	$4.79$ &	$106$ \\
		$10^{-4}$ &	$1.64$	& $47$ &	$0.43$	& $20$ &	$8.16$ &	$144$ \\
		$10^{-5}$  & $2.01$	& $53$ &	$0.58$	& $24$ &	$12.32$ &	$180$ \\
		$10^{-6}$  & $2.56$	& $61$ &	$0.82$	& $31$ &	$16.72$ &	$212$ \\
	\end{tabular}
\end{table}

\subsubsection{Other Cauchy-Stieltjes functions}
Finally, we test the convergence rate of the different  pole selection strategies for the Cauchy-Stieltjes functions $\frac{1-e^{-\sqrt z}}{z}, z^{-0.2},z^{-0.8}$ and the matrix argument $A=\mathrm{trid}(-1,2,-1)$.

The results reported in Figure~\ref{fig:exp-cauchy-1d-functions} show that in all cases the poles from Corollary~\ref{cor:cauchy} and the extended Krylov method provide the best and the worst convergence rates, respectively. The EDS converges faster than \texttt{markovfunm} apart from the case of $z^{-0.2}$ where the two strategies perform similarly.
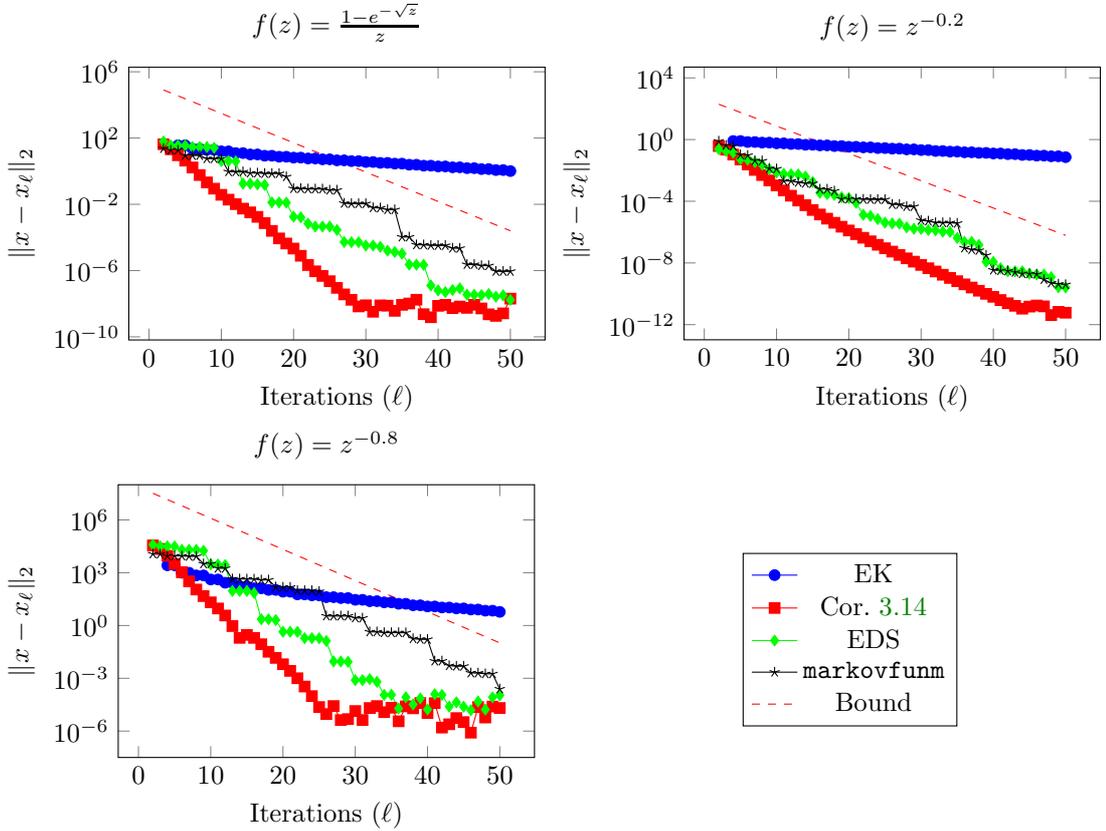
\begin{figure} 
	\begin{tikzpicture}
	\begin{semilogyaxis}[
	title = {$f(z) = \frac{1 - e^{-\sqrt{z}}}{z}$},
	legend style = {at={(1.05,1)}, anchor = north west}, 
	xlabel = Iterations ($\ell$), 
	ylabel = {$\norm{x  -x_{\ell}}_2$}, width = .48\linewidth, 
	height = .25\textheight,,ytick = {1e-10,1e-6, 1e-2, 1e2, 1e6}]			
	\addplot[blue, mark=*] table[x index = 0, y index = 1] {1d-50000-sqrtphi.dat};
	\addplot[red, mark=square*] table[x index = 0, y index = 2] {1d-50000-sqrtphi.dat};
	\addplot[green, mark=diamond*] table[x index = 0, y index = 3] {1d-50000-sqrtphi.dat};
	\addplot[black, mark=star] table[x index = 0, y index = 4] {1d-50000-sqrtphi.dat};				
	\addplot[red, dashed] table[x index = 0, y index = 5] {1d-50000-sqrtphi.dat};
	%\legend{Extended Krylov, Poles from Cor.~\ref{cor:cauchy}, \texttt{markovfunm}, Adaptive strategy}
	\end{semilogyaxis}
	\end{tikzpicture}~\begin{tikzpicture}
	\begin{semilogyaxis}[
	title = {$f(z) = z^{-0.2}$},
	legend style = {at={(1.05,1)}, anchor = north west}, 
	xlabel = Iterations ($\ell$), 
	ylabel = {$\norm{x  -x_{\ell}}_2$}, width = .48\linewidth, 
	height = .25\textheight,ytick = {1e-12,1e-8, 1e-4, 1e-0, 1e4},ymax=5e4]			
	\addplot[blue, mark=*] table[x index = 0, y index = 1] {1d-50000-z2.dat};
	\addplot[red, mark=square*] table[x index = 0, y index = 2] {1d-50000-z2.dat};
	\addplot[green, mark=diamond*] table[x index = 0, y index = 3] {1d-50000-z2.dat};
	\addplot[black, mark=star] table[x index = 0, y index = 4] {1d-50000-z2.dat};				
	\addplot[red, dashed] table[x index = 0, y index = 5] {1d-50000-z2.dat};
	%\legend{Extended Krylov, Poles from Cor.~\ref{cor:cauchy}, \texttt{markovfunm}, Adaptive strategy}
	\end{semilogyaxis}
	\end{tikzpicture} \\
	\begin{tikzpicture}
	\begin{semilogyaxis}[
	title = {$f(z) = z^{-0.8}$},
	legend style = {at={(1.5,.75)}, anchor = north west}, 
	xlabel = Iterations ($\ell$), 
	ylabel = {$\norm{x  -x_{\ell}}_2$}, width = .48\linewidth, 
	height = .25\textheight,ytick = {1e-6,1e-3, 1e0, 1e3, 1e6},ymax=1e8]			
	\addplot[blue, mark=*] table[x index = 0, y index = 1] {1d-50000-z8.dat};
	\addplot[red, mark=square*] table[x index = 0, y index = 2] {1d-50000-z8.dat};
	\addplot[green, mark=diamond*] table[x index = 0, y index = 3] {1d-50000-z8.dat};
	\addplot[black, mark=star] table[x index = 0, y index = 4] {1d-50000-z8.dat};				
	\addplot[red, dashed] table[x index = 0, y index = 5] {1d-50000-z8.dat};
	\legend{EK, Cor.~\ref{cor:cauchy}, EDS,  \texttt{markovfunm}, Bound}
	\end{semilogyaxis}
	\end{tikzpicture}
	\caption{Convergence history of the different projection spaces for the evaluation of $f(A)v$ for different Cauchy-Stieltjes
		functions $f(z)$ and the matrix argument$A=\mathrm{trid}(-1, 2, -1)$ of size
		$50000 \times 50000$. The methods tested are extended Krylov (EK), rational Krylov with the poles from 
		Corollary \ref{cor:cauchy}, rational Krylov with nested poles obtained as in Section~\ref{sec:nested} (EDS) and rational Krylov with the poles of \texttt{markovfunm}. The bound is obtained from Corollary~\ref{cor:cauchy}.}
	\label{fig:exp-cauchy-1d-functions}
\end{figure}
\section{Evaluating Stieltjes functions of matrices with Kronecker structure}
\label{sec:conv2d}

We consider the task of computing $f(\mathcal M) v$ where $\mathcal M = 
I \otimes A - B^T \otimes I$. This problem often stems from the discretizations of 2D differential equations, such as 
the matrix transfer method used for fractional diffusion equations \cite{yang2011novel}. 

We assume that $v = \vect(F)$, where $F = U_F V_F^T$ where $U_F$ and $V_F$ are tall and skinny matrices. For
instance, when $f(z) = z^{-1}$, this is equivalent to
solving the matrix equation $AX - XB = F$. It is well-known that, 
if the spectra of $A$ and $B$ are separated, then the low-rank property
is numerically inherited by $X$ \cite{Beckermann2019}. For more general functions
than $z^{-1}$, a projection scheme that preserves the Kronecker structure has been proposed  in \cite{benzi2017approximation} using polynomial Krylov methods. We
briefly review it in Section~\ref{sec:projection-scheme}. 
The method proposed in \cite{benzi2017approximation} uses
tensorized polynomial Krylov subspaces, so it is not well-suited when
$A$ and $B$ are ill-conditioned, as it
often happens discretizing differential operators. Therefore, we propose to replace the latter with a tensor product of rational Krylov subspaces and we provide
a strategy for the pole selection. This enables a faster convergence
and an effective scheme for the approximation of the action of such
matrix function in a low-rank format. 

The case of Laplace-Stieltjes functions, described 
in Section~\ref{sec:laplace}, follows easily by the analysis performed 
for the pole selection with a generic matrix $A$. The error analysis
for Cauchy-Stieltjes functions, presented in Section~\ref{sec:cauchy}, requires more care
and builds on the theory for the solution of Sylvester equations. 

\subsection{Projection methods that preserve Kronecker structure}
\label{sec:projection-scheme}

If $A, B$ are $n \times n$ matrices, applying the projection scheme 
described in Section~\ref{sec:stieltjes} 
requires to build an orthonormal basis $W$ for a (low-dimensional) subspace $\mathcal W\subseteq\mathbb C^{n^2}$, together with the projections of $W^* \mathcal M W = H$ and 
$v_{\mathcal W} = W^* v$. Then the action of $f(\mathcal M)$ on $v$ is approximated 
by:
\[
f(\mathcal M) v \approx W f(H) v_{\mathcal W}.
\]

The trick at the core of the
projection scheme proposed in \cite{benzi2017approximation} consists
in choosing a tensorized subspace of the form
$\mathcal W := \mathcal U \otimes \mathcal V$, spanned by an orthonormal basis of the form
$W = U \otimes V$, where $U$ and $V$ are orthonormal bases of 
$\mathcal U\subseteq\mathbb C^n$ and $\mathcal V\subseteq\mathbb C^n$, respectively. 
With this choice, the projection of $\mathcal M$ 
onto $\mathcal U \otimes \mathcal V$ retains the same structure, that is 
\[
(U \otimes V)^* \mathcal M (U \otimes V) = I \otimes A_U - B_V^T \otimes I,
\]
where $ A_{\mathcal U} = U^* A U$ and $B_{\mathcal V} = V^* B V$. 

Since in our case $v = \vect(F)$ and $F = U_F V_F^T$, this 
enables to exploit the low-rank structure as well. Indeed, the
projection of $F$ onto $\mathcal U \otimes \mathcal V$ can be written
as $v_{\mathcal W} = \vect(F_{\mathcal W})=\vect((U^* U_F) ( V_F^TV))$. 
The high-level structure of the procedure is sketched in 
Algorithm~\ref{alg:subspace}.

\begin{algorithm}
	\caption{Approximate $\vect^{-1}(f(\mathcal M) \vect(F))$}\label{alg:split}
	\begin{algorithmic}[1]
		\Statex{{\bf procedure}} KroneckerFun($f$, $A$, $B$, $U_F$, $V_F$)\Comment{Compute $f(\mathcal M) \vect(F)$}
		\State $\quad\!$ $U,V \gets $ \label{step:rhs} orthonormal bases for the selected subspaces. 
		\State $\quad\!$ $A_{\mathcal U} \gets U^* A U$
		\State $\quad\!$ $B_{\mathcal V} \gets V^* B V$
		\State $\quad\!$ $ F_{\mathcal W} \gets U^* U_F (V_F^TV)$
		\State $\quad\!$ $Y \gets \vect^{-1}(f(I \otimes  A_{\mathcal U}-B_{\mathcal V}^T \otimes I ) \vect( F_{\mathcal W}))$. \label{line:f(M)}
		\State $\quad$ \Return $UYV^*$
		\Statex{{\bf end procedure}} 
	\end{algorithmic}
	\label{alg:subspace}
\end{algorithm}

At the core of Algorithm~\ref{alg:subspace} is the evaluation of the
matrix function on the projected matrix $I \otimes A_{\mathcal U} - B_{\mathcal V}^T \otimes I$. 
Even when $U, V$ have a low dimension $k\ll n$, this matrix is $k^2 \times k^2$, 
so it is undesirable to build it explicitly and then evaluate $f(\cdot)$ on it. 

When $f(z) = z^{-1}$, it is well-known that such evaluation can be performed 
in $k^3$ flops by the Bartels-Stewart algorithm \cite{Bartels1972}, 
in contrast to
the $k^6$ complexity that would be required by a generic dense solver
for the system defined by $I \otimes A_{\mathcal U} - B_{\mathcal V}^T \otimes I$. For 
a more general function, we can still design a $\mathcal O(k^3)$ 
procedure for the evaluation of $f(\cdot)$ in our case. Indeed, since $A_{\mathcal U}$ and $B_{\mathcal V}$ are Hermitian, 
we may diagonalize them using a unitary transformation
as follows:
\[
Q_A^* A_{\mathcal U} Q_A = D_A, \qquad 
Q_B^* B_{\mathcal V} Q_B = D_B. 
\]
Then, the evaluation of the matrix function $f(z)$ 
with argument $I\otimes A_{\mathcal U} - B_{\mathcal V}^T\otimes I$ can be recast 
to a scalar problem 
by setting
\[
f(I\otimes A_{\mathcal U} - B_{\mathcal V}^T\otimes I) \mathrm{vec}(U^*F V) = 
\left( \overline Q_B \otimes {Q}_A \right) f(\mathcal D)
\left(Q_B^T \otimes {Q}_A^*\right)  \mathrm{vec}(U^*F V), 
\]
where $\mathcal D := I\otimes D_A - D_B \otimes I$. 
If we denote by $X = \vect^{-1}(f(\mathcal M) \mathrm{vec}(F))$ and
with $D$ the matrix  defined by $D_{ij} = (D_A)_{ii} - (D_B)_{jj}$, then
\[
X = Q_A \left[ f^\circ ( D ) \circ (Q_A^* U^*F V Q_B) \right] Q_B^*, 
\]
where $\circ$ denotes the Hadamard product and
$f^\circ(\cdot)$ the function $f(\cdot)$ applied
\emph{component-wise} to the entries of $D$: $[f^\circ(D)]_{ij} = f(D_{ij})$. 

Assuming that the matrices $Q_A, Q_B$ and the corresponding diagonal
matrices $D_A, D_B$, are available, this step  requires
$k^2$ scalar function evaluation, plus $4$ matrix-matrix
multiplications, for a total computational cost bounded
by $\mathcal O(c_f \cdot k^2 + k^3)$, where $c_f$ denotes the
cost of a single function evaluation. The procedure is
described in Algorithm~\ref{alg:diagonalization}.

\begin{algorithm}
	\caption{Evaluation of $f(I \otimes A_{\mathcal U} - B_{\mathcal V}^T \otimes I) \mathrm{vec}(U^*FV)$ 
		for normal $k \times k$
		matrices $A_{\mathcal U}$, $B_{\mathcal V}$}
	\begin{algorithmic}[1]
		\Procedure{funm\_diag}{$f$, $A_{\mathcal U}$, $B_{\mathcal V}$, $U^*FV$}
		\State $(Q_A, D_A) \gets \Call{Eig}{A_{\mathcal U}}$
		\State $(Q_B, D_B) \gets \Call{Eig}{B_{\mathcal V}}$
		\State $F_{\mathcal W} \gets Q_A^* U^*FV Q_B$
		\For{$i,j = 1,\ldots,n$}
		\State $X_{ij} \gets f((D_A)_{ii} + (D_B)_{jj})
		\cdot \left(F_{\mathcal W}\right)_{ij}$
		\EndFor
		\State \Return $\vect(Q_A X Q_B^*)$
		\EndProcedure 
	\end{algorithmic}
	\label{alg:diagonalization}
\end{algorithm}

\subsection{Convergence bounds for Laplace-Stieltjes functions of matrices with Kronecker structure}\label{sec:laplace}

The study of approximation methods for Laplace-Stieltjes functions 
is made easier by the following property of the matrix exponential: whenever $M,N$ commute, then $e^{M + N} = e^M e^N$. Since the matrices $B^T \otimes I$ and $I \otimes A$ commute,
we have 
\[
x = \vect(X)=f(\mathcal M) v=\int_0^{\infty}e^{-t\mathcal M} v\mu(t)\ dt= \vect\left(\int_0^\infty e^{-tA}U_FV_F^Te^{t B}\mu(t)\ dt\right).
\]
Consider projecting the matrix $\mathcal M$ onto a tensorized subspace
spanned by the Kronecker products of unitary matrices $U \otimes V$. This, 
combined with Algorithm~\ref{alg:subspace}, 
yields an approximation whose
accuracy is closely connected with the one of
approximating $e^{-tA}$ by projecting using $U$, and $e^{tB}$ using $V$. As discussed in Section~\ref{sec:stieltjes}, there exists a choice of poles that approximates
uniformly well the matrix exponential, and this can be leveraged here
as well. 

\begin{corollary} \label{cor:stieltjes-lapl-posdef}
	Let $f(z)$ be a Laplace-Stieltjes function, $A,-B$ be Hermitian positive definite with spectrum contained in $[a,b]$ and $X_\ell$ be the approximation of 
	$X = \vect^{-1}(f(\mathcal M)\vect(F))$ obtained using
	Algorithm~\ref{alg:subspace} with $U\otimes V$  orthonormal basis of $\mathcal{U_R}\otimes\mathcal{V_R}=\mathcal{RK}_\ell(A,U_F,\Pole_{\ell}^{[a,b]})\otimes\mathcal{RK}_\ell(B^T,V_F,\Pole_{\ell}^{[a,b]})$. Then, 
	\[
	\norm{X - X_\ell}_2 \leq 16\gamma_{\ell,\kappa}f(0^+)\rho_{[a,b]}^{\frac{\ell}{2}} 
	\norm{F}_2. 
	\]
\end{corollary}
\begin{proof}
	If $f(z)$ is a Laplace-Stieltjes function, we may express the error 
	matrix $X - X_\ell$ 
	as follows:
	\[
	X - X_\ell = \int_0^\infty 
	\left[ e^{-tA} F e^{tB} - Ue^{-tA_{\ell}} U^* F V e^{tB_{\ell}} V^* \right]\mu(t)  \ dt, 
	\]
	where $A_{\ell}=U^*AU$ and $B_{\ell}=V^*BV$.
	Adding and subtracting the quantity $Ue^{-tA_{\ell}} U^*F e^{tB}$ yields
	the following inequalities:
	\begin{equation*} 
	\begin{split}
	\norm{X-X_\ell}_2&\leq \int_0^\infty\norm{e^{-tA}F-Ue^{-tA_{\ell}}(U^*F)}_2 \norm{e^{tB}}_2 \mu(t)\ dt\\
	&+ \int_0^\infty\norm{e^{tB^T}F^T-Ve^{tB_{\ell}^T}(V^*F^T)}_2 \norm{e^{-tA_{\ell}}}_2\mu(t)\ dt\\
	&\leq 16\gamma_{\ell,\kappa}   \int_0^\infty    \mu(t)\ dt \cdot\rho_{[a,b]}^{\frac{\ell}{2}} \norm{F}_2
	\end{split}
	\end{equation*}
	where in the last step we used Corollary~\ref{cor:lapl} for both addends.
\end{proof}

\begin{example} \label{ex:lapl2d}
	To test the proposed projection spaces we consider the same matrix $A$ of Example~\ref{ex:lapl1d}, 
	and we evaluate the function $\varphi_1$ to $\mathcal M = I \otimes A + A \otimes I$, 
	applied to a vector $v = \vect(F)$, where $F$ is a random rank $1$ matrix, 
	generated by taking the outer product of two unit vectors with normally 
	distributed entries. The results are reported in Figure~\ref{fig:exp-lapl-2d}.
	\begin{figure}
		\begin{tikzpicture}
		\begin{semilogyaxis}[
		legend style={at={(1.05,1)}, anchor = north west},
		xlabel = Iterations ($\ell$), 
		ylabel = {$\norm{X  -X_{\ell}}_2$}, width = .6\linewidth, 
		height = .3\textheight]
		\addplot[no marks, dashed, red] table[x index = 0, y index = 1] {laplace_stieltjes_2D.dat};
		\addplot[blue, mark=*] table[x index = 0, y index = 2] {laplace_stieltjes_2D.dat};
		\addplot[brown, mark=triangle*] table[x index = 0, y index = 3]{laplace_stieltjes_2D.dat};
		\addplot[mark=square*, red] table[x index = 0, y index = 4]{laplace_stieltjes_2D.dat};
		\addplot[mark=diamond*, green] table[x index = 0, y index = 5]{laplace_stieltjes_2D.dat};
		\addplot[mark=o, black, only marks] table[x index = 0, y index = 7]{laplace_stieltjes_2D.dat};
		\legend{Bound from Cor.~\ref{cor:stieltjes-lapl-posdef}, 
			Extended Krylov, Polynomial Krylov, 
			Poles from Cor.~\ref{cor:stieltjes-lapl-posdef}, 
			EDS,
			Singular values of $X$};
		\end{semilogyaxis}
		\end{tikzpicture}
		\caption{Convergence history of the different projection spaces for the evaluation of  $\varphi_1(\mathcal M)v$ with
			the Kronecker structured matrix $\mathcal M = I \otimes A + A \otimes I$, 
			where $A$ is of size $1000 \times 1000$ and has condition number about $5 \cdot 10^5$.
			The singular values of the true solution $X$ are reported as well.}
		\label{fig:exp-lapl-2d}
	\end{figure}
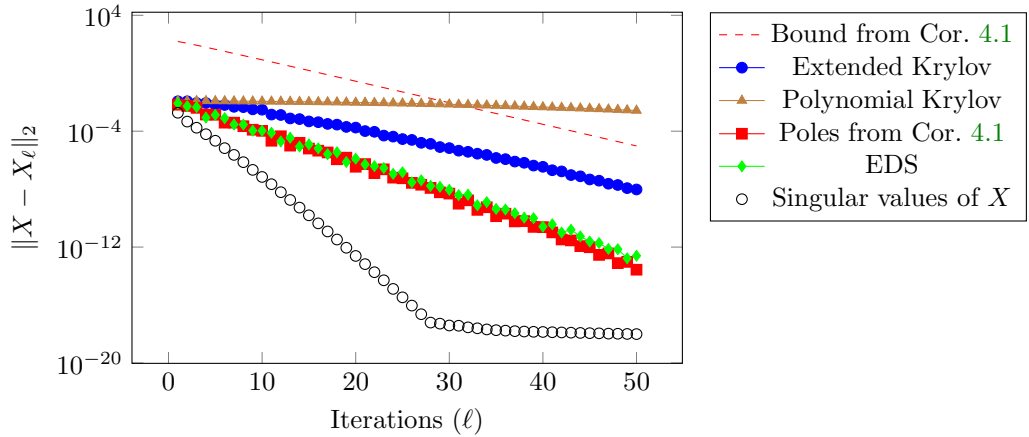
\end{example}	

\subsection{Convergence bounds for Cauchy-Stieltjes functions of matrices with Kronecker structure} \label{sec:cauchy}
As already pointed out in Section~\ref{sec:stieltjes}, evaluating a Cauchy-Stieltjes function requires a space which approximates 
uniformly well the shifted inverses of the matrix argument  under consideration. 
When considering a matrix $\mathcal M = I \otimes A - B^T \otimes I$ which
is Kronecker structured, this acquires a particular meaning. 

In fact, relying on the
integral representation \eqref{eq:stieltjes} of $f(z)$ we obtain:
\[
f(\mathcal M) v = \int_0^\infty \mu(t) (tI + \mathcal M)^{-1} \vect(F) \ dt
= \int_0^\infty \mu(t) X_t\ dt, 
\]
where $X_t := \vect^{-1}((tI + \mathcal M)^{-1} \vect(F))$ solves the matrix equation 
\begin{equation} \label{eq:sylvcauchy}
\left(t I + A\right) X_j - X_j B   = F.
\end{equation}
Therefore, to determine a good projection
space for the function evaluation, we should aim at determining a projection
space where these parameter dependent Sylvester equations can be solved
uniformly accurate. We note that, unlike in the Laplace-Stieltjes case, 
the evaluation of the resolvent does not split into the evaluation
of the shifted inverses of the factors, and this does not allow to 
apply  Theorem~\ref{thm:lapl-zol} for the factors $A$ and $B$. 

A possible strategy to determine an approximation space
is using polynomial Krylov subspaces 
$\mathcal{K}_m(t I + A, U_F) \otimes \mathcal{K}_m(B^T, V_F)$
for solving \eqref{eq:sylvcauchy} at a certain point $t$. 
Thanks to the shift invariance of 
polynomial Krylov subspaces, all these subspaces coincide with 
$\mathcal U_P \otimes \mathcal V_P=\mathcal{K}_m(A, U_F) \otimes \mathcal{K}_m(B^T, V_F)$. 
This observation is at the core of
the strategy proposed in \cite{benzi2017approximation}, which 
makes use of $\mathcal U_P \otimes \mathcal V_P$ in Algorithm~\ref{alg:subspace}. 
This allows 
to use the convergence theory for linear matrix equations to 
provide error bounds in the Cauchy-Stieltjes case, see \cite[Section 6.2]{benzi2017approximation}. 

Since rational Krylov subspaces are usually more effective in the solution
of Sylvester equations, it is natural to consider their use in place
$\mathcal U_P \otimes \mathcal V_P$. However, they are not shift invariant, 
and this makes the analysis not straightforward. 
Throughout this section,
we denote by $U\otimes V$ the orthonormal basis of the tensorized rational Krylov subspace
\begin{equation}\label{eq:spaces} 
\mathcal U_R\otimes \mathcal V_R := \mathcal{RK}_{\ell}(A, U_F, \Pole) \otimes \mathcal{RK}_{\ell}(B^T, V_F, \Xi)
\end{equation}
where $\Pole:=\{  \pole_1,\dots,\pole_\ell \}$ and 
$\Xi:=\{ \xi_1,\dots,\xi_\ell \}$ are the prescribed poles. 
We define the following polynomials of degree (at most) $\ell$: \begin{equation}\label{eq:poly}
p(z):=\prod_{j=1,\pole_j\neq \infty}^\ell (z-\pole_j), \qquad
q(z):=\prod_{j=1,\xi_j\neq \infty}^\ell (z-\xi_j)
\end{equation} 
and we denote by $A_{\ell}=U^*AU$,  $B_{\ell}=V^*BV$ the projected $(\ell k\times \ell k)$-matrices, where $k$ is the number of columns of $U_F$ and $V_F$.

In Section~\ref{sec:error-sylvester}, we recall and slightly extend some 
results about rational Krylov methods for Sylvester equations 
i.e., the case $f(z) = z^{-1}$. This will be the
building block for the convergence analysis of the approximation of Cauchy-Stieltjes 
functions in Section~\ref{sec:error-cauchy}. 

\subsubsection{Convergence results for Sylvester equations} 
\label{sec:error-sylvester}
Algorithm~\ref{alg:subspace} applied to $f(z) = z^{-1}$ coincides with
the Galerkin projection method for Sylvester equations \cite{simoncini2016computational}, whose error analysis can be found in  \cite{Beckermann2011}; the results in that paper relate the Frobenius norm
of the residual to a rational approximation problem.  
We state a slightly modified version of Theorem~2.1 in \cite{Beckermann2011}, that enables to bound the residual
in the Euclidean norm. The proof is reported in the Appendix~\ref{app:gal-res}.
\begin{theorem}\label{thm:gal-res}
	Let $A,-B$ be Hermitian positive definite with spectrum contained in $[a,b]$ and $X_\ell$ be the approximate solution returned by Algorithm~\ref{alg:subspace} using $f(z)=z^{-1}$ and the orthonormal 
	basis $U\otimes V$ of  $\mathcal U_R\otimes\mathcal V_R=\mathcal{RK}_{\ell}(A, U_F, \Pole) \otimes \mathcal{RK}_{\ell}(B^T, V_F, \Xi)$, then 
	\[
	\norm{AX_\ell - X_\ell B - F}_2\leq (1+ \kappa) \max\{\theta_\ell(I_A,I_B, \Pole),\ \theta_\ell(I_B,I_A, \Xi)\}\norm{F}_2.
	\] 
\end{theorem}

\begin{remark}
	Using the mixed norm inequality $\norm{AB}_F \leq \norm{A}_F \norm{B}_2$, one
	can state the bound in the Frobenius norm as well:
	\[
	\norm{AX_{G}^{(\ell)} - X_{G}^{(\ell)} B - F}_F \leq (1+ \kappa) \sqrt{  \theta_\ell^2(I_A,I_B, \Pole) + \theta_\ell^2(I_B,I_A, \Xi) } \cdot \norm{F}_F, 
	\]
	which is tighter than the one in \cite{Beckermann2011}.
	%,which
	%involves the Frobenius norm of the evaluated matrix functions. 
\end{remark}

For our analysis, it is more natural to bound the approximation error of the
exact solution $X$, instead of the norm of the residual. Since the residual is closely related with
the backward error of the underlying linear system, bounding the
forward error $\norm{X - X_\ell}_2$ causes the appearances
of an additional condition number. 

\begin{corollary}\label{cor:sylv-error}
	If $X_\ell$ is the approximate solution 
	of the linear matrix equation $AX - XB = F$
	returned by Algorithm~\ref{alg:subspace} as in Theorem~\ref{thm:gal-res}, 
	then
	\[
	\norm{X_\ell - X}_2 \leq \frac{a +b}{2a^2}  \max\{\theta_\ell(I_A,I_B, \Pole),\ \theta_\ell(I_B,I_A, \Xi)\}\norm{F}_2
	\]
\end{corollary}

\begin{proof}
	We note that $X_\ell - X$ solves the Sylvester equation
	$
	A(X_\ell - X) - (X_\ell - X) B = R, 
	$
	where $R := AX_\ell - X_\ell B - F$
	verifies $\norm{R}_2 \leq \left(1+\kappa\right) \max\{\theta_\ell(I_A,I_B, \Pole),\ \theta_\ell(I_B,I_A, \Xi)\}\norm{F}_2$, thanks
	to Theorem~\ref{thm:gal-res}. In view of \cite[Theorem~2.1]{ripjohn}
	$\norm{X_\ell - X}_2$ is bounded by
	$\frac{1}{2a} \norm{R}_2$. 
\end{proof}

\subsubsection{Error analysis for Cauchy-Stieltjes functions} \label{sec:error-cauchy} In view of Equation~\ref{eq:sylvcauchy}, 
the evaluation of Cauchy-Stieltjes function is closely
related to solving (in a uniformly accurate way) 
parameter-dependent Sylvester equations. This connection is 
clarified by the following result.

\begin{theorem} \label{thm:stieltjes-zolotarev}
	Let $f(z)$ be a Cauchy-Stieltjes function, 
	$A,-B$ be Hermitian positive definite with spectrum contained in $[a,b]$ and $X_\ell$ be the approximate evaluation of $f(z)$ returned by 
	Algorithm~\ref{alg:subspace} using the orthonormal basis $U\otimes V$ of the
	subspace $\mathcal U_R\otimes \mathcal V_R=\mathcal{RK}_{\ell}(A, U_F, \Pole) \otimes \mathcal{RK}_{\ell}(B^T, V_F, \Xi)$. Then,
	\begin{equation} \label{eq:cauchy-bound}
	\norm{X -X_\ell}_2 \leq f(2a) \cdot(1+\kappa)\cdot \norm{F}_2 \cdot \max_{t \geq 0} \left[ \max \Big\{\theta_\ell( I_A, I_B-t,\Pole),\ \theta_\ell(I_B,I_A+t,\Xi)\Big\}\right],
	\end{equation}
	where $\kappa=\frac ba$ and
	$\theta_\ell(\cdot,\cdot,\cdot)$ is as in Definition~\ref{def:theta}. 
\end{theorem}

\begin{proof}
	Applying the definition of $f(\mathcal M)$ we have 
	$
	f(\mathcal M) \vect{(F)} = 
	\int_0^\infty  (tI + \mathcal M)^{-1} \vect{(F)}\mu(t)\ dt. 
	$	
	We note that, for any $t \geq 0$, the vector $\vect(X_t) := (tI + \mathcal M )^{-1} \vect(F)$ is 
	such that $X_t$ solves the Sylvester equation 
	$(tI + A) X_t - X_t B = F$. Then, we can write $X$ as 
	$
	X = \int_0^\infty  X_t\mu(t)\ dt. 
	$
	
	Let us consider the approximation $U Y_t V^*$ to $X_t$ obtained by solving the projected
	Sylvester equation $(tI + U^* A U) Y_t - Y_t (V^* B V ) = U^* F V$, 
	and $Y = \int_0^\infty  Y_t\mu(t)\ dt$. 
	We remark that $\mathcal {RK}_\ell(A,U_F,\Pole)=\mathcal {RK}_\ell(tI + A,U_F,\Pole+t)$.
	
	Then, relying  on Corollary~\ref{cor:sylv-error}, 
	we can bound the error $R_t := \norm{X_t - UY_tV^*}_2$ with
	\[
	R_t \leq C(t) \cdot \max\left\{ 
	\theta_\ell\left(I_A+t, I_B, \Pole+t\right), 
	\theta_\ell\left(I_B, I_A+t, \Xi\right)
	\right\}\norm{F}_2, 
	\]
	where $C(t) := \frac{2(t+a+b)}{(t+2a)^2}$. Making use of 
	Lemma~\ref{lem:tp}-\ref{lem:tp:shift} we get:
	\begin{align*}
	R_t &\leq C(t) \cdot \underbrace{\max\left\{ 
		\theta_\ell\left(I_A, I_B-t, \Pole\right), 
		\theta_\ell\left(I_B, I_A+t, \Xi\right)
		\right\}}_{:= \Theta_\ell(t)}\norm{F}_2. 
	\end{align*}
	An estimate for the error on $X$ is obtained by integrating
	$R_t$: 
	\begin{align*}
	\norm{X- X_\ell}_2 &\leq \int_0^\infty \mu(t) \frac{2(t+a+b)\norm{F}_2}{(t+2a)^2}\Theta_\ell(t) dt \\
	&\leq (1+\kappa)\norm{F}_2\int_{0}^\infty \frac{\mu(t)}{t+2a} \Theta_\ell(t) dt  \\
	&\leq f(2a) \cdot(1+\kappa)\cdot \norm{F}_2 \cdot \max_{t\geq 0}\Theta_\ell(t), 
	\end{align*}
	where we used that 
	the function $\frac{2(t+a+b)}{t+2a}$ is maximum over $[0, \infty]$
	at $t=0$. 
\end{proof}

Inspired by Theorem~\ref{thm:gal-res}, we look at the construction
of rational functions that make the quantities $\theta_\ell(I_A,I_B-t,\Pole)$ and $\theta_\ell(I_B,I_A+t,\Xi)$ small. 
If we choose $\Xi=-\Pole$  then \eqref{eq:cauchy-bound} simplifies to
\begin{equation}\label{eq:simplify}
\norm{X -X_\ell}_2 \leq f(2a) \cdot(1+\kappa)\cdot \norm{F}_2 \cdot \max_{t \geq 0} \theta_\ell( I_A, I_B-t,\Pole),
\end{equation}
because $\theta_\ell( I_A, I_B-t,\Pole)=\theta_\ell( I_A, -I_A-t,\Pole)=\theta_\ell( -I_A, I_A+t,-\Pole)=\theta_\ell( I_B, I_A+t,-\Pole)$, in view of Lemma~\ref{lem:tp}-\ref{lem:tp:mobius}.

Similarly to the analysis done for Cauchy-Stieltjes function for a generic
matrix $A$, we may consider a M\"obius transform that maps the 
Zolotarev problem involving the point at infinity in a more familiar form. 
More precisely, we aim at mapping the set $[-\infty, -a] \cup [a, b]$ into $[-1, -\widetilde a]\cup[\widetilde a, 1]$ --- for some $\widetilde a\in (0,1)$. Then,
we make use of Theorem~\ref{thm:zolotarev} and Lemma~\ref{lem:tp}-\ref{lem:tp:mobius} to provide a  choice of $\Pole$ that makes the quantity 
$\theta_\ell(I_A, I_B-t, \Pole)$ small, independently of $t$. 

\begin{lemma}  \label{lem:mobius}
	The M\"obius transformation \[
	T(z):=\frac{\Delta + z - b}{\Delta - z + b},  \qquad 
	\Delta := \sqrt{b^2 - a^2},
	\]
	maps $[-\infty,-a]\cup[a,b]$ into $[-1, -\widetilde a]\cup[\widetilde a, 1]$, with $\tilde a := \frac{\Delta + a - b}{\Delta - a + b}$. The inverse map
	$T(z)^{-1}$ is:
	\[
	T^{-1}(z):= \frac{(b+\Delta)z+b-\Delta}{1+z}.
	\]
	In addition, we have $\tilde a^{-1} \leq 2b/a$, and therefore 
	$\rho_{[\tilde a, 1]} \leq \rho_{[a, 2b]}$. 
\end{lemma}
\begin{proof}
	The proof can be easily obtained following the same steps of Lemma~\ref{lem:mobius1d}. As in that case, the overestimate
	introduced by the inequality $\rho_{[\tilde a, 1]} \leq \rho_{[a, 2b]}$
	is negligible in practice (see Remark~\ref{rem:estimate-cauchy-exponent}). 
\end{proof}

In light of the previous result, we consider Theorem~\ref{thm:stieltjes-zolotarev} with the choice of poles 
\begin{equation} \label{eq:cauchy-poles}
\Pole = \Pole_{C_2,\ell}^{[a,b]} := T^{-1}(\Pole_\ell^{[\widetilde a,1]}), \qquad 
\Xi=-\Pole_{C_2,\ell}^{[a,b]}, 
\end{equation}
where $\Pole_\ell^{[\widetilde a,1]}$ indicates the set of optimal poles and zeros --- provided by Theorem~\ref{thm:zolotarev} --- for the domain $[-1,\widetilde a] \cup [\widetilde a,1]$. This yields the following.  

\begin{corollary} \label{cor:stieltjes-zolotarev-posdef}
	Let $f(z)$ be a Cauchy-Stieltjes function with density $\mu(t)$, 
	$A,-B$ be Hermitian positive definite with spectrum contained in $[a,b]$ and $X_\ell$ the approximate evaluation of $f(z)$ returned by 
	Algorithm~\ref{alg:subspace} using the orthonormal basis $U\otimes V$ of the
	subspace 
	$\mathcal{RK}_{\ell}(A, U_F, \Pole_{C_2,\ell}^{[a,b]})) 
	\otimes\mathcal{RK}_{\ell}(B^T, V_F, - \Pole_{C_2,\ell}^{[a,b]}))$,
	where $\Pole_{C_2, \ell}^{[a,b]}$ is as in \eqref{eq:cauchy-poles}. 
	Then,  		
	\begin{equation*}
	\norm{X - X_\ell}_2 \leq 4 \cdot f(2a) \cdot(1+\kappa)\cdot \norm{F}_2 \cdot \rho_{[a, 2b]}^\ell, \qquad 
	\rho_{[a, 2b]} := \exp\left(-\frac{\pi^2}{\log\left(\frac{8b}{a}\right)}\right). 
	\end{equation*}
\end{corollary}

\begin{proof}
	By setting $I_A = I$, $I_B = -I$ 
	in the statement of Theorem~\ref{thm:stieltjes-zolotarev} we get \eqref{eq:simplify}, so that we just need the bound  
	\begin{align*}
	\theta_\ell(I_A, I_B-t,T^{-1}(\Pole_\ell^{[\widetilde a,1]})) & = \theta_\ell(I_A, -I_A-t,T^{-1}(\Pole_\ell^{[\widetilde a,1]}))\leq \theta_\ell(I_A, [-\infty, -a],T^{-1}(\Pole_\ell^{[\widetilde a,1]}))\\ &= \theta_\ell([\widetilde a,1], [-1,-\widetilde a],\Pole_\ell^{[\widetilde a,1]})\leq 4\rho_{[\widetilde a,1]}^\ell, 
	\end{align*}
	where the first inequality follows from Lemma~\ref{lem:tp}-\ref{lem:tp:inclusion} and the last equality from  Lemma~\ref{lem:tp}-\ref{lem:tp:mobius} applied with the map $T(z)$. 
	The claim follows combining this inequality  $\rho_{[\tilde a, 1]} \leq 
	\rho_{[a,2b]}$ from Lemma~\ref{lem:mobius}. 
\end{proof}

\begin{example} \label{ex:cauchy2d}
	We consider the same matrix $A$ of Example~\ref{ex:cauchy1d}, 
	and we evaluate the inverse square root of $\mathcal M = I \otimes A + A \otimes I$, 
	applied to a vector $v = \vect(F)$, where $F$ is a random rank $1$ matrix, 
	generated by taking the outer product of two unit vectors with normally 
	distributed entries. 
	\begin{figure}
		\begin{tikzpicture}
		\begin{semilogyaxis}[legend style={at={(1.05,1)}, anchor = north west}, xlabel = Iterations ($\ell$), 
		ylabel = {$\norm{X  -X_\ell}_2$}, width = .5\linewidth, 
		height = .3\textheight]
		\addplot[no marks, dashed, red] table[x index = 0, y index = 1] {cauchy_stieltjes_2D.dat};
		\addplot[blue, mark=*] table[x index = 0, y index = 2] {cauchy_stieltjes_2D.dat};
		\addplot[brown, mark=triangle*] table[x index = 0, y index = 3]{cauchy_stieltjes_2D.dat};
		\addplot[mark=square*, red] table[x index = 0, y index = 4]{cauchy_stieltjes_2D.dat};
		\addplot[mark=diamond*, green] table[x index = 0, y index = 6]{cauchy_stieltjes_2D.dat};
		\addplot[mark=o, black, only marks] table[x index = 0, y index = 8]{cauchy_stieltjes_2D.dat};
		\addplot[dashdotted, blue] table[x index = 0, y index = 7]{cauchy_stieltjes_2D.dat};
		\legend{Bound from Cor.~\ref{cor:stieltjes-zolotarev-posdef}, 
			Extended Krylov, Polynomial Krylov, 
			Poles from Cor.~\ref{cor:stieltjes-zolotarev-posdef}, 
			EDS,
			Singular values of $X$, Bound for sing. values (Thm~\ref{thm:singdecay})};
		\end{semilogyaxis}
		\end{tikzpicture}
		\caption{Convergence history of the different projection spaces for the evaluation of  $\mathcal M^{-\frac 12}v$ with
			the Kronecker structured matrix $\mathcal M = I \otimes A + A \otimes I$, 
			where $A$ is of size $1000 \times 1000$ and has condition number about $5 \cdot 10^5$.
			The singular values of the true solution $X$ and the bound given in Theorem~\ref{thm:singdecay} are reported as well.}
		\label{fig:exp-cauchy-2d}
	\end{figure}
	We note that, in Figure~\ref{fig:exp-cauchy-2d}, the bound from Corollary~\ref{cor:stieltjes-zolotarev-posdef} accurately 
	predicts the asymptotic convergence rate, even though it is off by a constant; 
	we believe that this is due to the artificial introduction of $(1 + \kappa)$ in the
	Galerkin projection bound, which is usually very pessimistic in practice \cite{Beckermann2011}. 
\end{example}

\subsection{Low-rank approximability of \texorpdfstring{$X$}{X}}
\label{sec:low-rank-approximability}
The Kronecker-structured
rational Krylov method that we have discussed provides a practical way to compute the
evaluation of the matrix function under consideration. However, it can be used 
also theoretically to predict the decay in the singular values of the computed matrix
$X$, and therefore to describe its approximability properties in a low-rank format.

\subsubsection{Laplace-Stieltjes functions}
In the Laplace-Stieltjes case, we may employ Corollary~\ref{cor:stieltjes-lapl-posdef}
directly to provide an estimate for the decay in the singular values. 

\begin{theorem} \label{thm:singvals-lapl-decay}
	Let $f(z)$ be a Laplace-Stieltjes function and $\mathcal M = I \otimes A - B^T \otimes I$ where $A,-B$ are Hermitian positive definite with spectra contained in $[a, b]$. Then, 
	if $\vect(X) = f(\mathcal M) \vect(F)$, with $F = U_F V_F^T$ of rank $k$, we 
	have 
	\[
	\sigma_{1 + \ell k}(X) \leq  16 \gamma_{\ell,\kappa} f(0^+)\rho_{[a,b]}^{\frac{\ell}{2}} 
	\norm{F}_2. 
	\]
\end{theorem}
\begin{proof}
	We note that the approximation $X_\ell$ obtained using the rational Krylov method
	with the poles given by Corollary~\ref{cor:stieltjes-lapl-posdef} has rank (at most)
	$\ell k$, and $\norm{X- X_\ell}_2 \leq 16 \gamma_{\ell,\kappa} f(0^+)\rho_{[a,b]}^{\frac{\ell}{2}}$. 
	The claim follows by applying the Eckart-Young theorem. 
\end{proof}

\subsubsection{Cauchy-Stieltjes functions}
In the case of Cauchy-Stieltjes function, the error estimate  in Corollary~\ref{cor:stieltjes-zolotarev-posdef} would provides a result
completely analogue to Theorem~\ref{thm:singvals-lapl-decay}. However, the bound obtained this way involves the multiplicative factor $1+\kappa$; this can be avoided relying on an alternative strategy. 

The idea is to consider the close connection between the rational problem \eqref{eq:zol3} and the approximate solution returned by the \emph{factored Alternating Direction Implicit method} (fADI) \cite{Beckermann2011,Beckermann2019,benner}.
More specifically, for $t\geq 0$ let us denote with $X_t$, the solution of the shifted Sylvester equation 
\begin{equation}\label{eq:shift-eq}
(tI + A)X_t-X_tB=U_FV_F^*.
\end{equation}
In view of \eqref{eq:sylvcauchy}, $X_t$ is such that $X=\int_0^\infty X_t\mu(t)dt$. Running  fADI for $\ell$ iterations, with shift parameters $T^{-1}(\Pole_\ell^{[\widetilde a,1]})=\{\alpha_1,\dots,\alpha_\ell\}$ and $T^{-1}(-\Pole_\ell^{[\widetilde a,1]})=\{\beta_1,\dots,\beta_\ell\}$, provides an approximate solution $X^{ADI}_{\ell}(t)$ of \eqref{eq:shift-eq} such that its 
column and row span belong to the spaces 
\begin{align*}
\mathcal U_{\ell}^{\text{ADI}}(t) &= \mathcal{RK}(A, U_F, \{ \alpha_1 - t, \ldots, 
\alpha_{\ell} - t \}), &
\mathcal V_{\ell}^{\text{ADI}} &= \mathcal{RK}(B^T, V_F, \{ \beta_1, \ldots, 
\beta_{\ell} \}).
\end{align*}
Note that the space $V_{\ell}^{\text{ADI}}$ does not depend on $t$ because the right
coefficient of \eqref{eq:shift-eq} does not depend on $t$. If we denote by
$U_{\ell}^{\text{ADI}}(t)$ and $V_{\ell}^{\text{ADI}}$ orthonormal bases for these spaces, 
we have $X^{\text{ADI}}_{\ell}(t) = U_{\ell}^{\text{ADI}}(t) Y_{\ell}^{\text{ADI}}(t) (V_{\ell}^{\text{ADI}})^*$, 
and using the ADI error representation \cite{Beckermann2011,Beckermann2019} we obtain 
$
\norm{X_t - X^{\text{ADI}}_{\ell}(t)}_2 \leq \norm{X_t}_2  \rho_{[a, 2b]}^{\ell}. 
$

In particular, $X^{ADI}_{\ell}(t)$ is a uniformly good approximation of $X_t$ having rank (at most) $\ell k$ and  its low-rank factorization  has the same right factor $\forall t\geq 0$. 
\begin{theorem}\label{thm:singdecay}
	Let $f(z)$ be a Cauchy-Stieltjes function and $X = \vect^{-1}(f(\mathcal M) \vect(F))$, 
	with $\mathcal M := I \otimes A - B^T\otimes I$, where $A,-B$ are Hermitian positive definite
	with spectra contained in $[a, b]$. Then the singular values $\sigma_j(X)$ of the matrix $X$ verifies: 
	\[
	\sigma_{1+\ell k}(X) \leq 4  f(2a)  \rho_{[a,2b]}^\ell\norm{F}_2.
	\]
\end{theorem}
\begin{proof}
	Let us define $\widehat X_{\ell}:=\int_0^\infty X^{\text{ADI}}_{\ell}(t)\mu(t)dt=\int_0^\infty U^{\text{ADI}}_{\ell}(t)\mu(t)dt\cdot  Y^{\text{ADI}}_{\ell}(V^{\text{ADI}}_{\ell})^T$. Since 
	$V^{\text{ADI}}_{\ell}$ does not depend on $t$ we can take it out from
	the integral, and therefore $\widehat X_{\ell}$ has rank bounded by $\ell k$. Then, applying the Eckart-Young theorem we have the inequality
	\begin{align*}
	\sigma_{1+\ell s}(X) &\leq \norm{X-\widehat X_{\ell}}_2 \leq\int_0^\infty \norm{X_t-X^{\text{ADI}}_{\ell}(t)}_2\mu(t)dt 
	\leq 4 
	\int_0^\infty \rho_{[a,2b]}^\ell\norm{X_t}_2\mu(t)dt\\
	&\leq 4\int_0^\infty \frac{\mu(t)}{(t+2a)}dt\  \rho_{[a,2b]}^\ell \norm{F}_2
	=4f(2a)\rho_{[a,2b]}^\ell\norm{F}_2.
	\end{align*}
\end{proof}

\section{Conclusions, possible extensions and open problems}
\label{sec:conclusions}

We have presented a pole selection strategy for the rational Krylov methods when approximating  the action of
Laplace-Stieltjes and Cauchy-Stieltjes matrix functions on a vector. 
The poles have been shown to provide a fast convergence rate and explicit error bounds have been established. The theory of equidistributed sequences
has been used to obtained a nested sequence of poles with the same asymptotic convergence rate. Then,  the approach presented in \cite{benzi2017approximation}  that addresses the case of a matrix argument with a Kronecker sum structure has been extended to use rational Krylov subspaces. We have proposed a pole selection strategy that ensures a good exponential rate of convergence of the error norm. From the theoretical perspective we established decay bounds for the singular values of $\vect^{-1}(f(I\otimes A-B^T\otimes I)\vect(F))$ when $F$ is low-rank. This generalizes the well known low-rank approximability property of the solutions of Sylvester equations with low-rank right hand side. Also in the Kronecker
structured case, it has been shown that relying on equidistributed sequences
is an effective practical choice.

There are some research lines that naturally stem from this work. For instance, we have assumed for simplicity to be working with
Hermitian positive definite matrices. This assumption might be relaxed, by considering non-normal matrices with field of values included 
in the positive half plane. Designing an optimal pole selection for such 
problems would require the solution of Zolotarev problems on 
more general domains, and deserves further study. In addition, since the
projected problem is also non-normal, the fast diagonalization approach
for the evaluation proposed in Section~\ref{sec:projection-scheme} might 
not be applicable or stable, and therefore an alternative approach would need to 
be investigated. 

\section*{Acknowledgment}

The author wish to thank Paul Van Dooren and André Ran for 
fruitful discussions about Lemma~\ref{lem:laplinv}. 

\appendix 
\section{Proof of Theorem~\ref{thm:gal-res}}\label{app:gal-res}
According to \cite[Theorem~2.1]{Beckermann2011}, the residue $R := AX_\ell - X_\ell B - F$ can be written\footnote{In the original statement
	of \cite[Theorem~2.1]{Beckermann2011} the residual is decomposed in
	three parts; the missing term is equal to zero whenever the projection
	subspace contains the right hand side, which is indeed our case.} as
$\rho = \rho_{12} + \rho_{21}$, with
\begin{align*}
\rho_{12} &= U\cdot  {r_B^G}(A_{\ell})^{-1} \cdot F\cdot  r_B^G(B) &
\rho_{21} &= r_A^G(A) \cdot F \cdot {r_A^G}(B_{\ell})^{-1} V^*,
\end{align*}
where $r_A^G(z) := \det(zI - A_{\ell} ) / p(z)$, and 
$r_B^G(z) = \det(zI - B_{\ell}) / q(z)$, with $p(z),q(z)$ defined as in \eqref{eq:poly}. In addition, it is shown that $\rho_{12}=UU^*\rho(I-VV^*)$ and $\rho_{21}=(I-UU^*)\rho VV^*$. 

Moreover,
the proof of \cite[Theorem~2.1]{Beckermann2011} shows that, for 
any choice of $(\ell,\ell)$-rational function $r_B(x)$ with poles $z_1, \ldots, z_\ell$, we can further 
decompose $\rho_{12}$ as
$\rho_{12} = U (J_1 - J_2)$, 
where
\begin{align*}
J_1& = \frac{1}{2\pi\im}\int_{\Gamma_A}(zI-A_{\ell})^{-1}U^*F\cdot \frac{r_B(B)}{r_B(z)}V^*dz,\\
J_2&=\mathcal S_{A_\ell,B}\left( -\frac{1}{2\pi\mathbf i}\int_{\Gamma_A}(zI-A_{\ell})^{-1}U^*FV(zI-B_\ell)^{-1} \frac{r_B(B_\ell)}{r_B(z)}V^*dz\right),
\end{align*}
with $\mathcal S_{A,B}(X):=AX-XB$ and $\Gamma_A$  a path encircling once the interval $I_A$ but not $I_B$.

With a direct integration we get \[
J_1=r_B(A_\ell)^{-1}U^*F\cdot r_B(B)V^*,
\]
which yields $\norm{J_1}_2\leq \norm{F}_2\cdot \norm{r_B(A_\ell)^{-1}}_2\norm{r_B(B_\ell)}_2$. 		
Let $\widetilde B:= VB_\ell V^* -c(I-VV^*)$. Then,\footnote{The matrix
	$\tilde B$ is not used in the original proof of \cite{Beckermann2011}, which
	contains a minor typo. There, the operator $\mathcal S_{A_\ell,\tilde B}$ is 
	replaced by $\mathcal S_{A_\ell, B_\ell}$ which does not have compatible dimensions.}
\begin{align*}
\mathcal S_{A_\ell,\widetilde B}(\mathcal S_{A_\ell, B}^{-1}(J_2))&=	\mathcal S_{A_\ell,\widetilde B}\left(-\frac{1}{2\pi\im}\int_{\Gamma_A}(zI-A_{\ell})^{-1}U^*FV(zI-B_\ell)^{-1} \frac{r_B(B_\ell)}{r_B(z)}V^*dz\right)\\
&= -\frac{1}{2\pi\im}\int_{\Gamma_A}(A_\ell -zI)(zI-A_{\ell})^{-1}U^*FV(zI-B_\ell)^{-1} \frac{r_B(B_\ell)}{r_B(z)}V^*dz\\ &- \frac{1}{2\pi\mathbf i}\int_{\Gamma_A}(zI-A_{\ell})^{-1}U^*FV (zI-B_{\ell})^{-1}(zI-B_{\ell})\frac{r_B(B_\ell)}{r_B(z)}V^*dz\\
&= \frac{1}{2\pi \im} \int_{\Gamma_A}U^*FV (zI-B_{\ell})^{-1}\frac{r_B(B_\ell)}{r_B(z)}V^*dz \\
&-\frac{1}{2\pi \im} \int_{\Gamma_A}(zI-A_{\ell})^{-1}U^*FV \frac{r_B(B_\ell)}{r_B(z)}V^*dz \\ 
&=-\frac{1}{2\pi\im}\int_{\Gamma_A}(zI-A_{\ell})^{-1}U^*FV \frac{r_B(B_\ell)}{r_B(z)}V^*dz 
=-r_B(A_\ell)^{-1}U^*FVr_B(B_\ell)V^*,
\end{align*}
where we used that $V^*\widetilde B=B_\ell V^* $ and that the integral on the path $\Gamma_A$ of $(zI-B_\ell)^{-1}/r_B(z)$ vanishes.   	
Notice that $\norm{\mathcal S_{A,B}(X)}_2\leq (\norm{A}_2+\norm{B}_2)\norm{X}_2$ and $\norm{\mathcal S_{A,B}^{-1}(X)}_2\leq \norm{X}_2/\min_{i,j}|\lambda_i(A)-\lambda_j(B)|$ \cite[Theorem~2.1]{ripjohn}. We get
$
\norm{J_2}_2\leq \kappa\norm{r_B(A_\ell)^{-1}}_2\norm{r_B(B_\ell)}_2\norm{F}_2
$
and consequently
\[
\norm{\rho_{12}}\leq \norm{J_1}_2+\norm{J_2}_2\leq \left(1+\kappa\right)\norm{r_B(A_\ell)^{-1}}_2\norm{r_B(B_\ell)}_2\norm{F}_2.
\]
Taking the minimum over all $(\ell,\ell)$-rational functions with poles $\Xi$ provides
$
\norm{\rho_{12}}_2\leq(1+\kappa)\theta_\ell(I_B,I_A,\Xi)\norm{F}_2.
$
Analogously one obtains the similar estimate for $\rho_{21}$ swapping the role of $A$ and $B$. Since $\rho_{12}$ and $\rho_{21}$ have orthonormal rows and columns, we have $\norm{\rho_{12}+\rho_{21}}_2=\max\{\norm{\rho_{12}}_2,\norm{\rho_{21}}_2\}$, which concludes the proof. 

\section{Bounding an inverse Laplace transform}\label{app:lapl} The proof 
of Theorem~\ref{thm:lapl-zol} requires to bound the infinity
norm of an inverse Laplace transform of a particular
rational function, given in Lemma~\ref{lem:laplinv}. The purpose of this appendix is to
provide the details of its proof, that uses elementary arguments even though it is quite long. 

Let us consider the following functions, usually called 
\emph{sine integral} functions, that will be useful in the
following proofs:
\[
\Si(x) := \int_0^{x} \frac{\sin(t)}{t}\ dt, \qquad 
\si(x) := \int_x^{\infty} \frac{\sin(t)}{t}\ dt. 
\]
It is known that $\si(x) + \Si(x) = \frac{\pi}{2}$, and 
that $0 \leq \Si(x) \leq 1.852$ (see \cite[Section~6.16]{abramowitz1965handbook}), 
and therefore $|\si(x)| \leq \frac{\pi}{2}$. We will need
the following result, which involved integral of the 
sinc function by some particular measure. 

\begin{lemma} \label{lem:boundsinc}
	Let $g(t)$ be a decreasing and positive
	$\mathcal C^1$ 
	function over an
	interval $[0, \gamma]$. Then, the following inequality
	holds:
	\[
	\left|\int_0^\gamma \frac{\sin(s) g(s)}{s}\ ds \right|\leq
	1.852 \cdot g(0). 
	\]
\end{lemma}
\begin{proof}
	Integrating by parts yields 
	$
	I = \Si(s) g(s) \Big|_0^\gamma - 
	\int_0^\gamma \Si(s) g'(s)\ ds.
	$
	The first term is equal to $\Si(\gamma) g(\gamma)$, 
	which can be 
	bounded by $1.852 \cdot g(\gamma)$. 
	The second part can be bounded in modulus
	with
	\[
	\left|\int_0^\gamma \Si(s) g'(s)\ ds\right| \leq 
	- \max_{[0, \gamma]} |\Si(s)| \cdot \int_0^\gamma g'(s)\ ds = 
	(g(0) - g(\nu)) \max_{[0, \nu]} |\Si(s)|, 
	\]
	where we have used that $g'(s)$ is negative, so $|g'(s)| = -g'(s)$. Combining the two inequalities we have 
	\[
	|I| \leq 1.852 \cdot g(\gamma) + 1.852 \cdot (g(0) - g(\gamma)) = 
	1.852 \cdot g(0). 
	\]
\end{proof}
Given a 
set of positive real points $\gamma_j$ enclosed in a 
interval $[a, b]$ with $a > 0$, we define the
rational function 
\begin{equation} \label{eq:rs}
\widehat r(s) := \frac{1}{s} \frac{p(s)}{p(-s)}. 
\end{equation}
Note that $\widehat r(s)$ has poles enclosed in the negative
half-plane which ensures that  $\lim_{t\to\infty}\mathcal L^{-1}[\widehat r(s)]=1$. In particular $\mathcal L^{-1}[\widehat r(s)]$ is bounded on $\mathbb R_+$. 
We are now ready to prove Lemma~\ref{lem:laplinv}.

\begin{proof}[Proof of Lemma~\ref{lem:laplinv}]
	We write
	the inverse Laplace transform as follows:
	\[
	f(t) = \frac{1}{2\pi \im} \lim_{T \to \infty}\int_{-\im T}^{\im T}
	\widehat r(s) e^{st} ds. 
	\]
	The integration path needs to be chosen to keep all
	the poles on its left, including zero. Therefore, we 
	choose the path $\gamma_\epsilon$ that goes from
	$-\im T$ to $-\im \epsilon$, follows a semicircular path around $0$ 
	on the right, and then connects $\im \epsilon$ to $\im T$. It is sketched
	in the following figure:
	
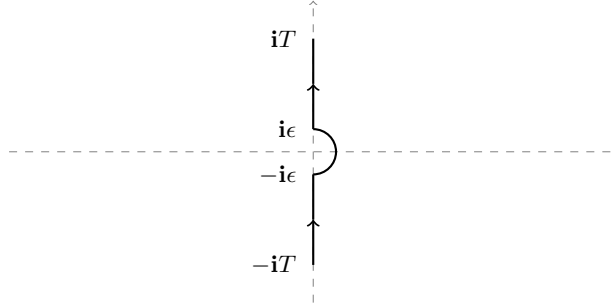
\begin{center}
		\begin{tikzpicture}
		\draw[dashed, gray, very thin, ->] (-4, 0) -- (4, 0);
		\draw[dashed, gray, very thin, ->] (0, -2) -- (0, 2);
		\draw[thick,->] (0,-1.5) -- (0,-0.9);
		\draw[thick] (0,-0.9) -- (0,-0.3);
		\draw[thick,->] (0,0.3) -- (0,0.9);
		\draw[thick] (0,0.9) -- (0,1.5);
		\draw[thick] (0,-0.3) arc (-90:90:0.3);
		\node at (-.1,-1.5) [anchor=east]{\small $-\im T$};
		\node at (-.1,1.5) [anchor=east]{\small $\im T$};
		\node at (-.1,-0.3) [anchor=east]{\small $-\im \epsilon$};
		\node at (-0.1,0.3) [anchor=east] {\small $\im \epsilon$};
		\end{tikzpicture}
	\end{center}
	Splitting the integral in the three segments we obtain the
	formula:
	\begin{equation} \label{eq:splitint}
	f(t) = \frac{1}{2\pi \im} \int_{\partial B(0,\epsilon) \cap \mathbb C_+} \widehat r(s) e^{st} \ ds 
	+ \lim_{T \to \infty} \left(
	\frac{1}{2\pi \im} \int_{-\im T}^{\im \epsilon} \widehat r(s) e^{st} ds + 
	\frac{1}{2\pi \im} \int_{\im \epsilon}^{\im T} \widehat r(s) e^{st} ds \right). 
	\end{equation}
	Concerning the first term, it is immediate to see that the integrand
	uniformly converges to the $1/s$ for $\epsilon \to 0$, 
	and therefore the first terms goes to $\frac{1}{2}$ for
	$\epsilon$ small. We now focus on the second part. 
	
	We can rephrase the ratio of polynomials defining 
	$r(s)$ as follows:
	\[
	\frac{p(s)}{p(-s)}=\prod_{j=1}^\ell\frac{\gamma_j-s}{\gamma_j+s},\qquad \gamma_j\in[a,b],\quad 0<a<b.
	\]	
	Then, we note that the above ratio restricted to the points
	of the form $\im s$ yields a complex number of modulus one, 
	that must have the form $\frac{p(\im s)}{p(-\im s)} = e^{\im \theta(s)}$, 
	where
	\begin{align*}
	\theta(s):=\arg\left(\frac{p(\im s)}{p(-\im s)}\right)&=\sum_{j=1}^\ell\arg(\gamma_j-\im s)-\arg(\gamma_j+\im s)
	=-2\sum_{j=1}^\ell\atan\left(\frac{s}{\gamma_j}\right)\in[-\ell\pi,\ell\pi].
	\end{align*}	
	In particular, $\lim_{s\to\infty}\theta(s)=-\ell\pi$ and for $s> 0$ it holds
	\begin{align}
	\ell\pi +\theta(s)&= \sum_{j=1}^\ell 2\left(\frac{\pi}{2}- \atan\left(\frac{s}{\gamma_j}\right)\right)
	=\sum_{j=1}^\ell 2\left(\int_0^\infty\frac{1}{1+x^2}dx-\int_0^{\frac{s}{\gamma_j}}\frac{1}{1+x^2}dx\right)\\
	&=2\sum_{j=1}^\ell \int_{\frac{s}{\gamma_j}}^\infty\frac{1}{1+x^2}dx\leq 2\sum_{j=1}^\ell \int_{\frac{s}{\gamma_j}}^\infty\frac{1}{x^2}dx
	=2\frac{\sum_{j=1}^\ell\gamma_j}s\leq \frac{2\ell b}{s}. \label{eq1}
	\end{align}	
	This allows to rephrase the integrals of \eqref{eq:splitint}
	in the more convenient form
	\[
	\frac{1}{2\pi \im} \int_{\im \epsilon}^{\im T}\widehat r(s) e^{st}\ ds  = 
	\frac{1}{2\pi \im} \int_{\epsilon}^{T} \im \cdot \widehat r(\im s) e^{\im st}\ ds = 
	\frac{(-1)^\ell}{2\pi \im} \int_{\epsilon}^{T} \frac{e^{\im(st + \theta(s))}}{s} \ ds.
	\]
	Since we are summing the integral between $[\epsilon, \infty]$
	and $[-\infty, \epsilon]$ we can drop the odd part of the
	integrand, and rewrite their sum as follows:
	\[
	\frac{(-1)^\ell}{2\pi \im} \int_{\epsilon}^{T} \frac{e^{\im(st + \theta(s))}}{s} \ ds + 	  \frac{(-1)^\ell}{2\pi \im} \int_{-T}^{\epsilon} \frac{e^{\im(st + \theta(s))}}{s} \ ds = 
	\frac{(-1)^\ell}{\pi} \int_{\epsilon}^{T} \frac{\sin(st + \theta(s))}{s} \ ds.
	\]
	The above integral is well-defined even if we let 
	$\epsilon \to 0$, we can can take the limit in \eqref{eq:splitint} which yields exactly the value 
	$\frac{1}{2}$ for the first term, and we have reduced the
	problem to estimate 
	$
	f(t) = \frac{1}{2} + \frac{(-1)^\ell}{\pi} \int_{0}^{\infty} \frac{\sin(st + \theta(s))}{s} \ ds.
	$
	To bound the integral, we split the integration domain in
	three parts:
	\begin{align*}
	\frac 1\pi\int_0^{\infty}\frac{\sin(st +\theta(s))}{s}ds = &\underbrace{\frac 1\pi\int_0^{\nu}\frac{\sin(st +\theta(s))}{s}ds}_{I_1}+  
	\underbrace{\frac 1\pi\int_\nu^{\xi}\frac{\sin(st +\theta(s))}{s}ds}_{I_2}+\underbrace{\frac 1\pi\int_\xi^{\infty}\frac{\sin(st +\theta(s))}{s}ds}_{I_3},
	\end{align*}
	where we choose $\nu = a \tan(\frac{\pi}{4\ell})$ 
	and $\xi = 4\ell b$. 
	For later use, we note that $\frac{a\pi}{4\ell} \leq \nu \leq \frac{a}{\ell}$. 
	Concerning $I_1$, we can further split the integral as
	$
	I_1 =\frac 1\pi\int_0^{\nu}\frac{\sin(st) \cos(\theta(s))}{s}ds+\frac 1\pi\int_0^{\nu}\frac{\cos(st) \sin(\theta(s))}{s}ds.
	$
	Note that $|\theta(s)| \leq 2s \sum_{j = 1}^\ell \gamma_j^{-1}$, 
	which can be obtained making use of the inequality 
	$|\atan(t)| \leq |t|$. We can bound the second integral
	term as follows:
	\[
	\frac 1\pi \left|\int_0^{\nu}\frac{\cos(st) \sin(\theta(s))}{s}ds\right|\leq 
	\frac 1\pi \int_0^{\nu}\frac{\cos(st) |\theta(s)|}{s}ds 
	\leq \nu \frac{1}{\pi} \sum_{j = 1}^\ell \gamma_j^{-1} 
	\leq \frac{1}{\pi}, 
	\]
	where we have used $\nu \leq \frac{a}{\ell}$ and 
	$\sum_j^\ell \gamma_j^{-1} \leq \frac{\ell}{a}$. The first part
	can be bounded making use of Lemma~\ref{lem:boundsinc}, by
	introducing the change of variable $y = st$, which yields 
	\[
	\frac{1}{\pi} \int_0^\nu \frac{\sin(st) \cos(\theta(s))}{s} ds
	= \frac{1}{\pi} \int_0^{t \nu} \frac{\sin(y) \cos(\theta(y/t))}{y} dy. 
	\]
	Note that on $[0, t \nu]$ the function $\cos(\theta(y/t))$
	is indeed decreasing, thanks to our choice of $\nu$, and
	therefore the above can be bounded in modulus by 
	$
	\frac{1}{\pi} \left|\int_0^{t \nu} \frac{\sin(y) \cos(\theta(y/t))}{y} dy\right| \leq \frac{1.852}{\pi}, 
	$
	where we have used that $\cos(\theta(0)) = 1$, and applied 
	Lemma~\ref{lem:boundsinc}. 
	
	Concerning $I_2$ we have 
	\[
	|I_2|=\left|\frac 1\pi\int_\nu^{\xi}\frac{\sin(st +\theta(s))}{s}ds\right|\leq \frac 1\pi\int_\nu^{\xi}\frac{1}{s}ds=\frac{1}{\pi}\log\left(\frac{\xi}{\nu}\right) \leq 
	\frac{1}{\pi}\log\left(\frac{16\ell^2 b}{a \pi}\right) 
	\]
	Concerning $I_3$, we perform the same splitting 
	for a since of a sum that we had for $I_1$, yielding 
	\[
	I_3= \underbrace{\frac 1\pi\int_\xi^{\infty}\frac{\sin(st)\cos( \theta(s))}{s}ds}_{I_4} + \underbrace{\frac 1\pi\int_\xi^{\infty}\frac{\cos(st)\sin( \theta(s))}{s}ds}_{I_5}.
	\]
	By using \eqref{eq1} we have that  $\forall s\in[\xi,\infty)$:
	\[
	\cos(\theta(s))= \cos(-\ell\pi + \varphi(s)) = (-1)^\ell \cos(\varphi(s)),\qquad 0\leq \varphi(s)\leq \frac{2\ell b}{s}.
	\]
	Using the Lagrange expression for the residual of the Taylor expansion we get
	$
	\cos(\varphi(s))\hspace{-10pt}\underbrace{=}_{\psi\in[0,\varphi(s)]}\hspace{-10pt}1-\sin(\psi(s))\varphi(s).
	$
	This enables bounding $I_4$ as follows:
	\begin{align*}
	|I_4|&= \frac 1\pi\left|\int_\xi^{\infty}\frac{\sin(st)\cos( \theta(s))}{s}ds\right|
	\leq  \frac 1\pi\left|\int_\xi^{\infty}\frac{\sin(st)}{s}ds\right| + \frac 1\pi\int_\xi^{\infty}\left|\frac{\sin(st)\sin(\psi(s))\varphi(s)}{s}\right|ds\\
	&\leq  \frac 1\pi\left|\int_\xi^{\infty}\frac{\sin(st)}{s}ds\right| + \frac 1\pi\int_\xi^{\infty}\frac{2\ell b}{s^2}ds
	=\frac{1}{\pi}\left(|\si(\xi)| + \frac{2kb}{\xi}\right) \leq 
	\frac{1}{2} + \frac{2\ell b}{\xi \pi} \leq \frac{1}{2} + \frac{1}{2\pi}.
	\end{align*}
	Analogously, for bounding $I_5$, we remark that
	by using \eqref{eq1} we have that  $\forall s\in[\xi,\infty)$:
	\[
	\sin(\theta(s))= \sin(-\ell \pi + \varphi(s)) = (-1)^\ell \sin(\varphi(s)), \qquad 0\leq \varphi(s)\leq \frac{2\ell b}{s}.
	\]
	Hence, 
	\begin{align*}
	|I_5| &\leq \frac{1}{\pi} \left| \int_\xi^\infty \frac{\cos(st) \sin(\theta(s))}{s}\ ds \right| \leq 
	\frac{1}{\pi} \int_\xi^\infty \frac{|\sin(\varphi(s))|}{s}\ ds 
	\leq \frac{1}{\pi} \int_\xi^\infty \frac{|\varphi(s)|}{s}\ ds \leq \frac{2kb}{\xi \pi} \leq \frac{1}{2\pi}. 
	\end{align*}
	Combining all these inequalities, we have
	\begin{align*}
	\norm{f(t)}_{L^\infty(\mathbb R_+)} &\leq 
	\frac{1}{2} + \frac{1}{\pi} + \frac{1.852}{\pi} + \frac{1}{\pi}\log\left(16 \ell^2 \frac{b}{a \pi}\right) + \frac{1}{2} + 
	\frac{1}{\pi} 
	\leq 2.23 + \frac{2}{\pi} \log\left(
	4\ell \cdot \sqrt{\frac{b}{\pi a}}
	\right). 
	\end{align*}
\end{proof}	

\bibliographystyle{siamplain}
\bibliography{library_new}

\begin{thebibliography}{10}

\bibitem{abramowitz1965handbook}
{\sc M.~Abramowitz and I.~A. Stegun}, {\em Handbook of mathematical functions:
  with formulas, graphs, and mathematical tables}, vol.~55, Courier
  Corporation, 1965.

\bibitem{Bartels1972}
{\sc R.~H. Bartels and G.~W. Stewart}, {\em {A}lgorithm 432: {S}olution of the
  {M}atrix {E}quation {$AX+XB=C$}}, Comm. {ACM}, 15 (1972), pp.~820--826.

\bibitem{Beckermann2011}
{\sc B.~Beckermann}, {\em An error analysis for rational {G}alerkin projection
  applied to the {S}ylvester equation}, SIAM J. Numer. Anal., 49 (2011),
  pp.~2430--2450, \url{https://doi.org/10.1137/110824590},
  \url{https://doi.org/10.1137/110824590}.

\bibitem{beckermann2009error}
{\sc B.~Beckermann and L.~Reichel}, {\em Error estimates and evaluation of
  matrix functions via the {F}aber transform}, SIAM J. Numer. Anal., 47 (2009),
  pp.~3849--3883, \url{https://doi.org/10.1137/080741744},
  \url{https://doi.org/10.1137/080741744}.

\bibitem{Beckermann2019}
{\sc B.~Beckermann and A.~Townsend}, {\em Bounds on the singular values of
  matrices with displacement structure}, SIAM Rev., 61 (2019), pp.~319--344,
  \url{https://doi.org/10.1137/19M1244433},
  \url{https://doi.org/10.1137/19M1244433}.
\newblock Revised reprint of "On the singular values of matrices with
  displacement structure" [ MR3717820].

\bibitem{benner}
{\sc P.~Benner, R.-C. Li, and N.~Truhar}, {\em On the {ADI} method for
  {S}ylvester equations}, J. Comput. Appl. Math., 233 (2009), pp.~1035--1045,
  \url{https://doi.org/10.1016/j.cam.2009.08.108},
  \url{https://doi.org/10.1016/j.cam.2009.08.108}.

\bibitem{benzi2013total}
{\sc M.~Benzi and C.~Klymko}, {\em Total communicability as a centrality
  measure}, Journal of Complex Networks, 1 (2013), pp.~124--149.

\bibitem{benzi2017approximation}
{\sc M.~Benzi and V.~Simoncini}, {\em Approximation of functions of large
  matrices with {K}ronecker structure}, Numer. Math., 135 (2017), pp.~1--26,
  \url{https://doi.org/10.1007/s00211-016-0799-9},
  \url{https://doi.org/10.1007/s00211-016-0799-9}.

\bibitem{berg}
{\sc C.~Berg}, {\em Stieltjes-{P}ick-{B}ernstein-{S}choenberg and their
  connection to complete monotonicity}, Positive Definite Functions: From
  Schoenberg to Space-Time Challenges,  (2008), pp.~15--45.

\bibitem{bernstein}
{\sc S.~Bernstein}, {\em Sur les fonctions absolument monotones}, Acta Math.,
  52 (1929), pp.~1--66, \url{https://doi.org/10.1007/BF02547400},
  \url{https://doi.org/10.1007/BF02547400}.

\bibitem{braess2012nonlinear}
{\sc D.~Braess}, {\em Nonlinear approximation theory}, vol.~7, Springer Science
  \& Business Media, 2012.

\bibitem{breiten2016low-rank}
{\sc T.~Breiten, V.~Simoncini, and M.~Stoll}, {\em Low-rank solvers for
  fractional differential equations}, Electron. Trans. Numer. Anal., 45 (2016),
  pp.~107--132.

\bibitem{druskin1998extended}
{\sc V.~Druskin and L.~Knizhnerman}, {\em Extended {K}rylov subspaces:
  approximation of the matrix square root and related functions}, SIAM Journal
  on Matrix Analysis and Applications, 19 (1998), pp.~755--771.

\bibitem{druskin2009}
{\sc V.~Druskin, L.~Knizhnerman, and M.~Zaslavsky}, {\em Solution of large
  scale evolutionary problems using rational {K}rylov subspaces with optimized
  shifts}, SIAM J. Sci. Comput., 31 (2009), pp.~3760--3780,
  \url{https://doi.org/10.1137/080742403},
  \url{https://doi.org/10.1137/080742403}.

\bibitem{druskin2010adaptive}
{\sc V.~Druskin, C.~Lieberman, and M.~Zaslavsky}, {\em On adaptive choice of
  shifts in rational {K}rylov subspace reduction of evolutionary problems},
  SIAM Journal on Scientific Computing, 32 (2010), pp.~2485--2496.

\bibitem{druskin2011adaptive}
{\sc V.~Druskin and V.~Simoncini}, {\em Adaptive rational {K}rylov subspaces
  for large-scale dynamical systems}, Systems Control Lett., 60 (2011),
  pp.~546--560, \url{https://doi.org/10.1016/j.sysconle.2011.04.013},
  \url{https://doi.org/10.1016/j.sysconle.2011.04.013}.

\bibitem{frommer2014efficient}
{\sc A.~Frommer, S.~G\"{u}ttel, and M.~Schweitzer}, {\em Efficient and stable
  {A}rnoldi restarts for matrix functions based on quadrature}, SIAM J. Matrix
  Anal. Appl., 35 (2014), pp.~661--683,
  \url{https://doi.org/10.1137/13093491X},
  \url{https://doi.org/10.1137/13093491X}.

\bibitem{guttel2013rational}
{\sc S.~G\"{u}ttel}, {\em Rational {K}rylov approximation of matrix functions:
  numerical methods and optimal pole selection}, GAMM-Mitt., 36 (2013),
  pp.~8--31, \url{https://doi.org/10.1002/gamm.201310002},
  \url{https://doi.org/10.1002/gamm.201310002}.

\bibitem{guettel2013blackbox}
{\sc S.~G\"{u}ttel and L.~Knizhnerman}, {\em A black-box rational {A}rnoldi
  variant for {C}auchy-{S}tieltjes matrix functions}, BIT, 53 (2013),
  pp.~595--616, \url{https://doi.org/10.1007/s10543-013-0420-x},
  \url{https://doi.org/10.1007/s10543-013-0420-x}.

\bibitem{hochbruck2010exponential}
{\sc M.~Hochbruck and A.~Ostermann}, {\em Exponential integrators}, Acta
  Numer., 19 (2010), pp.~209--286,
  \url{https://doi.org/10.1017/S0962492910000048},
  \url{https://doi.org/10.1017/S0962492910000048}.

\bibitem{ripjohn}
{\sc R.~A. Horn and F.~Kittaneh}, {\em Two applications of a bound on the
  {H}adamard product with a {C}auchy matrix}, Electron. J. Linear Algebra, 3
  (1998), pp.~4--12, \url{https://doi.org/10.13001/1081-3810.1010},
  \url{https://doi.org/10.13001/1081-3810.1010}.
\newblock Dedicated to Hans Schneider on the occasion of his 70th birthday.

\bibitem{kalugin2012stieltjes}
{\sc G.~A. Kalugin, D.~J. Jeffrey, R.~M. Corless, and P.~B. Borwein}, {\em
  Stieltjes and other integral representations for functions of {L}ambert {W}},
  Integral Transforms and Special Functions, 23 (2012), pp.~581--593.

\bibitem{kressner2019low}
{\sc D.~Kressner, S.~Massei, and L.~Robol}, {\em Low-rank updates and a
  divide-and-conquer method for linear matrix equations}, SIAM J. Sci. Comput.,
  41 (2019), pp.~A848--A876, \url{https://doi.org/10.1137/17M1161038},
  \url{https://doi.org/10.1137/17M1161038}.

\bibitem{massei2019fast}
{\sc S.~Massei, M.~Mazza, and L.~Robol}, {\em Fast solvers for two-dimensional
  fractional diffusion equations using rank structured matrices}, SIAM J. Sci.
  Comput., 41 (2019), pp.~A2627--A2656,
  \url{https://doi.org/10.1137/18M1180803},
  \url{https://doi.org/10.1137/18M1180803}.

\bibitem{massei2017solving}
{\sc S.~Massei, D.~Palitta, and L.~Robol}, {\em Solving rank-structured
  {S}ylvester and {L}yapunov equations}, SIAM J. Matrix Anal. Appl., 39 (2018),
  pp.~1564--1590, \url{https://doi.org/10.1137/17M1157155},
  \url{https://doi.org/10.1137/17M1157155}.

\bibitem{oseledets2007lower}
{\sc I.~V. Oseledets}, {\em Lower bounds for separable approximations of the
  {H}ilbert kernel}, Mat. Sb., 198 (2007), pp.~137--144,
  \url{https://doi.org/10.1070/SM2007v198n03ABEH003842},
  \url{https://doi.org/10.1070/SM2007v198n03ABEH003842}.

\bibitem{schweitzer2016restarting}
{\sc M.~Schweitzer}, {\em Restarting and error estimation in polynomial and
  extended {K}rylov subspace methods for the approximation of matrix
  functions}, PhD thesis, Universit{\"a}tsbibliothek Wuppertal, 2016.

\bibitem{simoncini2016computational}
{\sc V.~Simoncini}, {\em Computational methods for linear matrix equations},
  SIAM Rev., 58 (2016), pp.~377--441, \url{https://doi.org/10.1137/130912839},
  \url{https://doi.org/10.1137/130912839}.

\bibitem{susnjara2015accelerated}
{\sc A.~Susnjara, N.~Perraudin, D.~Kressner, and P.~Vandergheynst}, {\em
  Accelerated filtering on graphs using {L}anczos method}, arXiv preprint
  arXiv:1509.04537,  (2015).

\bibitem{townsend2015automatic}
{\sc A.~Townsend and S.~Olver}, {\em The automatic solution of partial
  differential equations using a global spectral method}, J. Comput. Phys., 299
  (2015), pp.~106--123, \url{https://doi.org/10.1016/j.jcp.2015.06.031},
  \url{https://doi.org/10.1016/j.jcp.2015.06.031}.

\bibitem{widder1942laplace}
{\sc D.~V. Widder}, {\em The {L}aplace {T}ransform}, Princeton Mathematical
  Series, v. 6, Princeton University Press, Princeton, N. J., 1941.

\bibitem{yang2011novel}
{\sc Q.~Yang, I.~Turner, F.~Liu, and M.~Ili\'{c}}, {\em Novel numerical methods
  for solving the time-space fractional diffusion equation in two dimensions},
  SIAM J. Sci. Comput., 33 (2011), pp.~1159--1180,
  \url{https://doi.org/10.1137/100800634},
  \url{https://doi.org/10.1137/100800634}.

\bibitem{zolotarev1877application}
{\sc E.~Zolotarev}, {\em Application of elliptic functions to questions of
  functions deviating least and most from zero}, Zap. Imp. Akad. Nauk. St.
  Petersburg, 30 (1877), pp.~1--59.

\end{thebibliography}

\end{document}